\documentclass{amsart}

\title{Infinite Hex is a draw}

\author{Joel David Hamkins}
\address[Joel David Hamkins]
{O'Hara Professor of Philosophy and Mathematics, University of Notre Dame, 100 Malloy Hall, Notre Dame, IN 46556 USA \&\ Associate Faculty Member, Professor of Logic, Faculty of Philosophy, 
Radcliffe Observatory Quarter 555, Woodstock Road, Oxford, OX2 6GG UK}
\email{jdhamkins@nd.edu}
\urladdr{http://jdh.hamkins.org}

\author{Davide Leonessi}
\address[Davide Leonessi]
{Program in Mathematics, the Graduate Center of the City University of New York, 365 Fifth Avenue, New York, NY 10016 USA}
\email{dleonessi@gc.cuny.edu}
\urladdr{http://leonessi.org}

\thanks{This article is adapted from chapter 2 of the second author's MSc dissertation~\cite{Leonessi2021:MSc-dissertation-transfinite-game-values-in-infinite-games}, for which he earned a distinction at the University of Oxford in September 2021. Commentary can be made about this article on the first author's blog at \href{http://jdh.hamkins.org/infinite-hex-is-a-draw}{http://jdh.hamkins.org/infinite-hex-is-a-draw}.}

\usepackage{booktabs,caption,subcaption}

\usepackage{latexsym,amsfonts,amsmath,amssymb,mathrsfs}
\usepackage{bbm}
\usepackage[hidelinks]{hyperref}
\usepackage{relsize}
\usepackage[rgb,dvipsnames,svgnames]{xcolor}
\usepackage{tikz}
\usetikzlibrary{arrows,arrows.meta,petri,topaths,positioning,shapes,shapes.misc,patterns,calc,decorations.pathreplacing,hobby}
\pgfdeclarelayer{foreground}
\pgfdeclarelayer{background}
\pgfdeclarelayer{boardarrows}
\pgfdeclarelayer{boardgrid}
\pgfdeclarelayer{boardshades}
\pgfsetlayers{boardshades,boardgrid,boardarrows,background,main,foreground}
\usepackage{wrapfig} 
\usepackage{caption} 
\DeclareCaptionStyle{compact}{font=footnotesize,format=plain,name={Fig},labelsep=period,skip=1.5ex,margin=0pt,oneside,justification=centering}
\DeclareCaptionStyle{leftside}{style=compact,margin={0pt,9pt}}
\DeclareCaptionStyle{rightside}{style=compact,margin={9pt,0pt}}
\usepackage{float}
\usepackage[utf8]{inputenc}
\RequirePackage{doi}

\usepackage{enumitem}

\newtheorem{theorem}{Theorem}
\newtheorem*{theorem*}{Theorem}

\newtheorem*{maintheorem*}{Main Theorem}
\newtheorem*{maintheorems*}{Main Theorems}
\newtheorem{corollary}[theorem]{Corollary}
\newtheorem*{corollary*}{Corollary}
\newtheorem*{corollaries*}{Corollaries}

\newtheorem{conjecture}[theorem]{Conjecture}

\theoremstyle{definition}

\newtheorem*{definition*}{Definition}

\newtheorem{question}[theorem]{Question}
\newtheorem*{question*}{Question}

\newtheorem*{questions*}{Questions}
\newtheorem*{mainquestion*}{Main Question} 
\newtheorem*{openquestion*}{Open Question} 
\theoremstyle{remark}

\newcommand{\QED}{\end{proof}}

\def\proclaim[#1]{{\bf #1}}
\def\BF#1.{{\bf #1.}}

\def\says#1:#2\par{\item[#1] #2\par}

\makeatletter
\newcommand{\dotminus}{\mathbin{\text{\@dotminus}}}
\newcommand{\@dotminus}{%
  \ooalign{\hidewidth\raise1ex\hbox{.}\hidewidth\cr$\m@th-$\cr}%
}
\makeatother
\renewcommand{\setminus}{\raise.3ex\hbox{\rotatebox{-20}{$-$}}} 

\newcommand{\smalllt}{\mathrel{\mathchoice{\raise2pt\hbox{$\scriptstyle<$}}{\raise1pt\hbox{$\scriptstyle<$}}{\raise0pt\hbox{$\scriptscriptstyle<$}}{\scriptscriptstyle<}}}
\newcommand{\smallleq}{\mathrel{\mathchoice{\raise2pt\hbox{$\scriptstyle\leq$}}{\raise1pt\hbox{$\scriptstyle\leq$}}{\raise1pt\hbox{$\scriptscriptstyle\leq$}}{\scriptscriptstyle\leq}}}

   \def\DHLhksqrt#1#2{%
   \setbox0=\hbox{$#1\sqrt{#2\,}$}\dimen0=\ht0
   \advance\dimen0-0.2\ht0
   \setbox2=\hbox{\vrule height\ht0 depth -\dimen0}%
   {\box0\lower0.4pt\box2}}
\newcommand{\boolval}[1]{\mathopen{\lbrack\!\lbrack}\,#1\,\mathclose{\rbrack\!\rbrack}}
\def\[#1]{\boolval{#1}}
\newbox\gnBoxA
\newbox\gnBoxB
\newdimen\gnCornerHgt
\setbox\gnBoxA=\hbox{\tiny$\ulcorner$}
\global\gnCornerHgt=\ht\gnBoxA
\newdimen\gnArgHgt
\def\gcode #1{%
\setbox\gnBoxA=\hbox{$#1$}%
\setbox\gnBoxB=\hbox{$\bar #1$}%
\gnArgHgt=\ht\gnBoxB%
\ifnum     \gnArgHgt<\gnCornerHgt \gnArgHgt=0pt%
\else \advance \gnArgHgt by -\gnCornerHgt%
\fi \raise\gnArgHgt\hbox{\tiny$\ulcorner$} \box\gnBoxA %
\raise\gnArgHgt\hbox{\tiny$\urcorner$}}
\newcommand{\UnderTilde}[1]{{\setbox1=\hbox{$#1$}\baselineskip=0pt\vtop{\hbox{$#1$}\hbox to\wd1{\hfil$\sim$\hfil}}}{}}
\newcommand{\Undertilde}[1]{{\setbox1=\hbox{$#1$}\baselineskip=0pt\vtop{\hbox{$#1$}\hbox to\wd1{\hfil$\scriptstyle\sim$\hfil}}}{}}
\newcommand{\undertilde}[1]{{\setbox1=\hbox{$#1$}\baselineskip=0pt\vtop{\hbox{$#1$}\hbox to\wd1{\hfil$\scriptscriptstyle\sim$\hfil}}}{}}
\newcommand{\UnderdTilde}[1]{{\setbox1=\hbox{$#1$}\baselineskip=0pt\vtop{\hbox{$#1$}\hbox to\wd1{\hfil$\approx$\hfil}}}{}}
\newcommand{\Underdtilde}[1]{{\setbox1=\hbox{$#1$}\baselineskip=0pt\vtop{\hbox{$#1$}\hbox to\wd1{\hfil\scriptsize$\approx$\hfil}}}{}}

\def\<#1>{\left\langle#1\right\rangle}

\newcommand{\cell}[1]{\boxit{\hbox to 17pt{\strut\hfil$#1$\hfil}}}
\newcommand{\head}[2]{\lower2pt\vbox{\hbox{\strut\footnotesize\it\hskip3pt#2}\boxit{\cell#1}}}
\newcommand{\boxit}[1]{\setbox4=\hbox{\kern2pt#1\kern2pt}\hbox{\vrule\vbox{\hrule\kern2pt\box4\kern2pt\hrule}\vrule}}
\newcommand{\Col}[3]{\hbox{\vbox{\baselineskip=0pt\parskip=0pt\cell#1\cell#2\cell#3}}}
\newcommand{\tapenames}{\raise 5pt\vbox to .7in{\hbox to .8in{\it\hfill input: \strut}\vfill\hbox to
.8in{\it\hfill scratch: \strut}\vfill\hbox to .8in{\it\hfill output: \strut}}}
\newcommand{\Head}[4]{\lower2pt\vbox{\hbox to25pt{\strut\footnotesize\it\hfill#4\hfill}\boxit{\Col#1#2#3}}}
\newcommand{\Dots}{\raise 5pt\vbox to .7in{\hbox{\ $\cdots$\strut}\vfill\hbox{\ $\cdots$\strut}\vfill\hbox{\
$\cdots$\strut}}}
\usepackage[backend=bibtex,style=alphabetic,maxbibnames=15,maxcitenames=6,dateabbrev=false]{biblatex}
\renewcommand{\UrlFont}{} 
\renewbibmacro{in:}{\ifentrytype{article}{}{\printtext{\bibstring{in}\intitlepunct}}} 
\DeclareFieldFormat{url}{{\UrlFont\url{#1}}} 
\DeclareFieldFormat{urldate}{
  (version \thefield{urlday}\addspace%
  \mkbibmonth{\thefield{urlmonth}}\addspace%
  \thefield{urlyear}\isdot)}
\DeclareFieldFormat{eprint:arxiv}{
\ifhyperref
    {\href{http://arxiv.org/abs/#1}{%
    arXiv\addcolon\nolinkurl{#1}}\iffieldundef{eprintclass}{}{\UrlFont{\mkbibbrackets{\thefield{eprintclass}}}}}
    {arXiv\addcolon\nolinkurl{#1}\iffieldundef{eprintclass}{}{\UrlFont{\mkbibbrackets{\thefield{eprintclass}}}}}}

\begin{document}

\begin{abstract}
We introduce the game of infinite Hex, extending the familiar finite game to natural play on the infinite hexagonal lattice. Whereas the finite game is a win for the first player, we prove in contrast that infinite Hex is a draw---both players have drawing strategies. Meanwhile, the transfinite game-value phenomenon, now abundantly exhibited in infinite chess and infinite draughts, regrettably does not arise in infinite Hex; only finite game values occur. Indeed, every game-valued position in infinite Hex is intrinsically local, meaning that winning play depends only on a fixed finite region of the board. This latter fact is proved under very general hypotheses, establishing the conclusion for all simple stone-placing games.
\end{abstract}
\maketitle

\begin{wrapfigure}[12]{r}{.26\textwidth}\vskip-5ex\hfill
  \includegraphics[width=.22\textwidth]{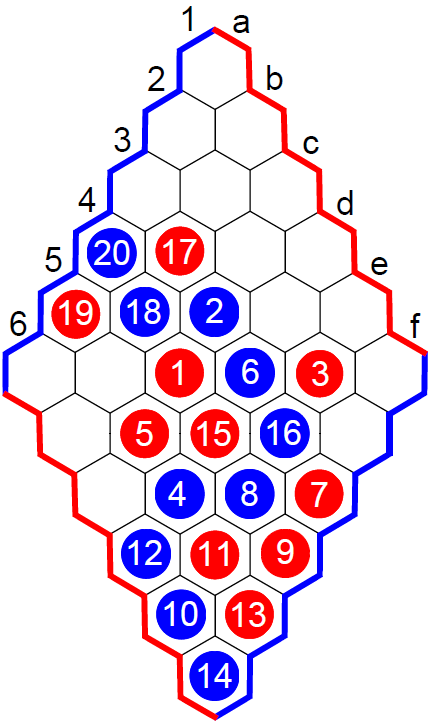}
  \captionsetup{style=rightside,font=footnotesize}
  \caption{The authors play a game of Hex}
  \label{Figure.Finite-Hex}
\end{wrapfigure}
\section{Introduction}
The game of Hex, popular amongst mathematicians, is played between two players, Red and Blue, who alternately place colored stones on a finite hexagonal grid, a rhombus, until one of them wins by connecting their sides of the board with stones of their color.\footnote{The game of Hex was invented in 1942 by polymath Piet Hein~\cite{hein42} under the name Polygon, becoming a popular board game in Denmark. The game was rediscovered by Nobel laureate John Nash while a student at Princeton in 1948; he described it as a ``matter of connecting topology and game theory''~\cite[\textsection 3]{hayward06}. Thus from the Princeton mathematics department common room, the game of Nash rose to popularity in the mathematical community, eventually published as the mainstream board game Hex in 1952.}
Blue has won the small game shown in Figure~\ref{Figure.Finite-Hex}, for example,
played just now by the authors, by connecting the two blue sides. The game of Hex offers enjoyably subtle strategic play---we highly recommend it---but it also has a very satisfactory mathematical analysis, involving applications of topology and graph theory in combinatorial game theory. The central facts can be summarized as follows.
\begin{theorem}[Finite Hex Theorem]\label{Theorem.Finite-hex}\ 
\begin{enumerate}
 \item\label{Theorem.Finite-hex.1} Every coloring of the finite rhombus Hex board with two colors exhibits a win for exactly one player. 
 \item\label{Theorem.Finite-hex.2} The game therefore admits a winning strategy for one of the players, and indeed it is the first player who has a winning strategy.
\end{enumerate}
\end{theorem}

For background, let us briefly sketch proofs for these well-known results. The main claim of statement~(\ref{Theorem.Finite-hex.1}) holds not only for the standard $n\times n$ rhombus-shaped\goodbreak

\begin{wrapfigure}{r}{.28\textwidth}
\hfill
\includegraphics[width=.25\textwidth]{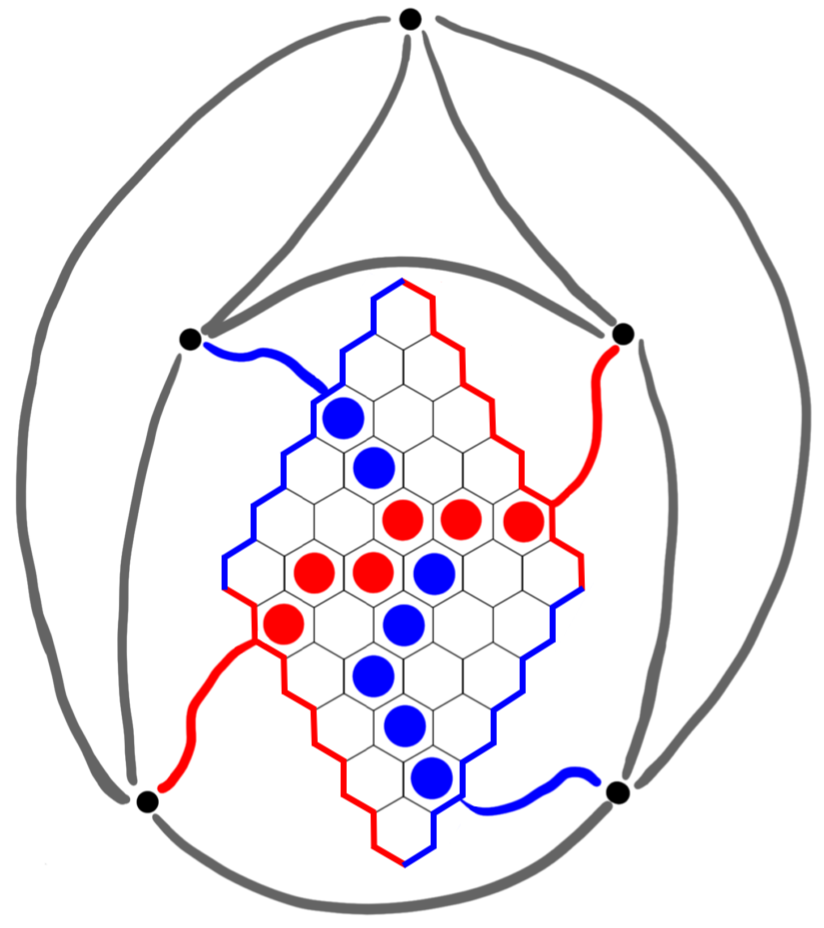}
\captionsetup{style=rightside,font=footnotesize}
\caption{Hex and $K_5$}
\label{Figure.K_5}
\end{wrapfigure}
\noindent Hex board, but more generally for any connected assemblage of hexagons whose boundary is divided into four contiguous segments in the pattern red, blue, red, blue, with each boundary segment overlapping on exactly one hexagonal tile with the next, just as with the four corners of the rhombus board. Any path connecting the two red segments will topologically separate the two blue segments from one another in light of the Jordan curve theorem, and so at most one player can win a given play of the game, establishing the uniqueness claim. Schachner~\cite{schachner19} argues alternatively that if both players had a winning path in a game of Hex, then we could exhibit the complete graph $K_5$ on five vertices as a planar graph, as illustrated in Figure~\ref{Figure.K_5}, but this is impossible.

Meanwhile, for the existence claim of Theorem~\ref{Theorem.Finite-hex} statement~(\ref{Theorem.Finite-hex.1}), the claim that every coloring of the board exhibits a win for one of the players, Gale~\cite{Gale79} argues by embarking on a direct path from the south following edges of the hexagons and remaining always on the boundary between a red hexagon on the left and a blue hexagon at the right, as in Figure~\ref{Figure.Gale-tour} (we may imagine augmenting the main board with additional colored stones of the hexagons just outside each of the four borders, using the color of that border). This path is uniquely determined, since as it enters any vertex it has a red hexagon on the left and blue at the right, and so it must turn left or right depending on the color of the new hexagon encountered at that junction. The resulting path must ultimately terminate at the nodes labeled $e$, $w$, or $n$, and in each
\begin{wrapfigure}{r}{.28\textwidth}
  \hfill
  \includegraphics[width=.25\textwidth]{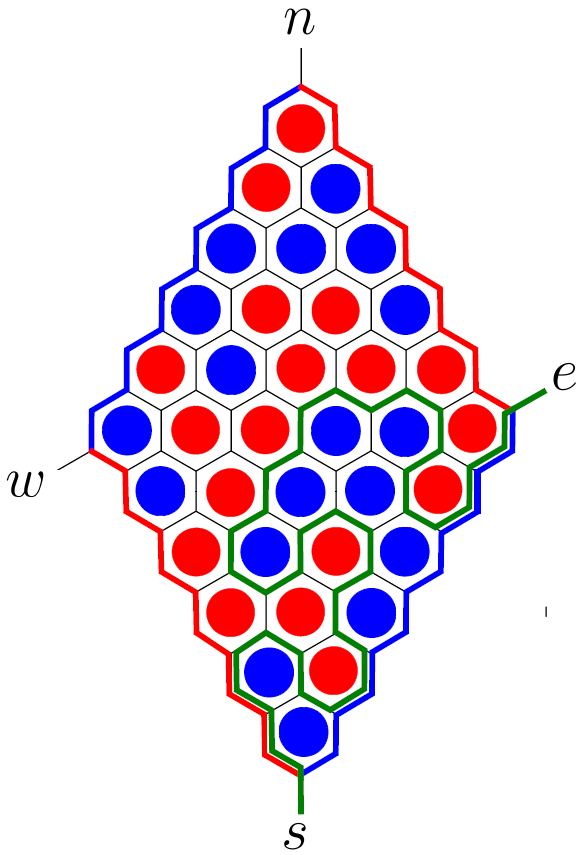}\hfill
  \captionsetup{style=rightside,font=footnotesize}
  \caption{Gale's tour}
  \label{Figure.Gale-tour}
\end{wrapfigure}
case one gets a winning path for one of the players by following the accompanying adjacent tiles on one side of the path or the other.\footnote{For both uniqueness and existence, these arguments use that the red and blue boundary segments overlap on exactly one hex tile with the adjacent segment. Namely, if the boundary segments did not have a tile in common, then one can exhibit drawn positions, where each player gets half the board, but there is no winner; and if they had overlapped by two cells, then there are positions in which both players have won.} Gale~\cite{Gale79} shows the theorem to be equivalent in a sense to the Brouwer Fixed Point Theorem, which in turn is shown equivalent to the Jordan Curve Theorem by~\cite{maehara84} and~\cite{adler16}.

The initial claim of Theorem~\ref{Theorem.Finite-hex} statement~(\ref{Theorem.Finite-hex.2}) now follows as an immediate consequence of (\ref{Theorem.Finite-hex.1}) by the fundamental theorem of finite games, which asserts that every finite game of perfect information admits a winning strategy for exactly one of the players (proved by Ernst Zermelo in 1913; see~\cite{Larson2010:Zermelo-1913}; see also~\cite[Section~7.7]{Hamkins2020:Proof-and-the-art-of-mathematics} for an elementary account). Namely, statement (\ref{Theorem.Finite-hex.1}) shows that every outcome of the game leads to a win for exactly one of the players in a uniformly bounded finite number of moves, and so this is a finite game covered by the fundamental theorem. Consequently, one of the players will have a winning strategy.

Finally, a strategy-stealing argument due to Nash~\cite{nash52} shows that the second player cannot have a winning strategy, and so it must be the first player who does so. Namely, if there would have been any advantage in doing so, player 1 can pretend to be player 2 by inventing an imaginary first move for player 2, responding just as a given strategy for player 2 would respond, but swapping the colors and reflecting the board on the center vertical, a transformation that exactly swaps the winning configurations of the two players. (If the actual player 2 should ever actually play the initial imaginary move, then player 1 should simply invent another imaginary move for player 2 at that stage, and play accordingly.) In this way, player 1 can play as though he or she is player 2, and this would be a winning strategy for player 1, if the strategy had been winning for player 2, which is a contradiction, since not both players can have winning strategies. Therefore, player 2 cannot actually have a winning strategy, and so it must be the first player who has the winning strategy. This argument generalizes to any shape of board, provided it is symmetric in a way that supports the strategy-stealing argument. The argument relies on the monotonicity property of Hex---having extra stones on the board is never disadvantageous for a player---since in the argument we described a play by which player 1 wins despite having given an imaginary extra stone to player 2, which can only be advantageous to player 2.

\enlargethispage{25pt}
This completes our sketch of the proof of Theorem~\ref{Theorem.Finite-hex}. Notice that because Blue won the game in Figure~\ref{Figure.Finite-Hex}, it was an upset win---we leave it to the reader to find the flaw in the strategy of Red, who chalks the loss up to overconfidence in light of Theorem~\ref{Theorem.Finite-hex} when playing first on such a small board.\goodbreak

\section{Infinite Hex}

We shall place our main focus in this article on the game of \emph{infinite} Hex, played on the infinite hexagonal lattice. Starting from an empty board, players Red and Blue alternately place their colored stones on the board, claiming hexagonal tiles. 

\begin{figure}[H]
  \centering
  \includegraphics[width=.85\textwidth]{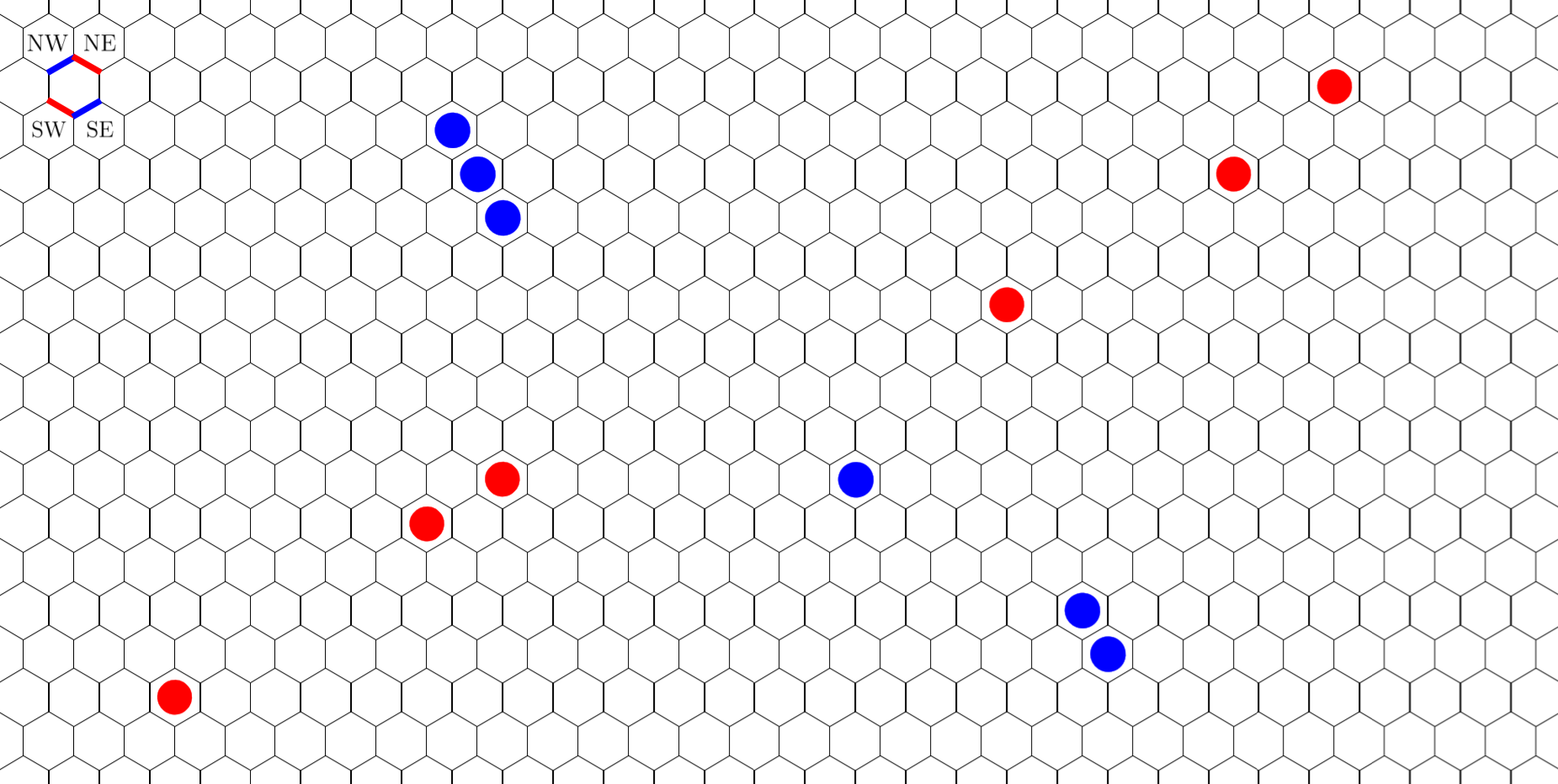}
  \caption{A partial play of infinite Hex}
  \label{Figure.partial-infinite-play}
\end{figure}

Play proceeds for infinitely many moves, after which a winner may be determined.

\subsection{The winning condition}

\enlargethispage{20pt}
But what is the winning condition exactly---how shall we determine who has won a play of infinite Hex? Since there are no boundary edges of the infinite board to be joined, it may not be clear initially exactly what condition entitles a player to be declared the winner. What is clear, however, is that we want to declare that Red wins, if there is a path of red hexagons stretching somehow from the lower left ``at infinity'' to the upper right ``at infinity.'' And similarly for Blue, but from the lower right to the upper left.\goodbreak 

\begin{wrapfigure}{r}{.35\textwidth}
\hfill
  \includegraphics[width=.33\textwidth]{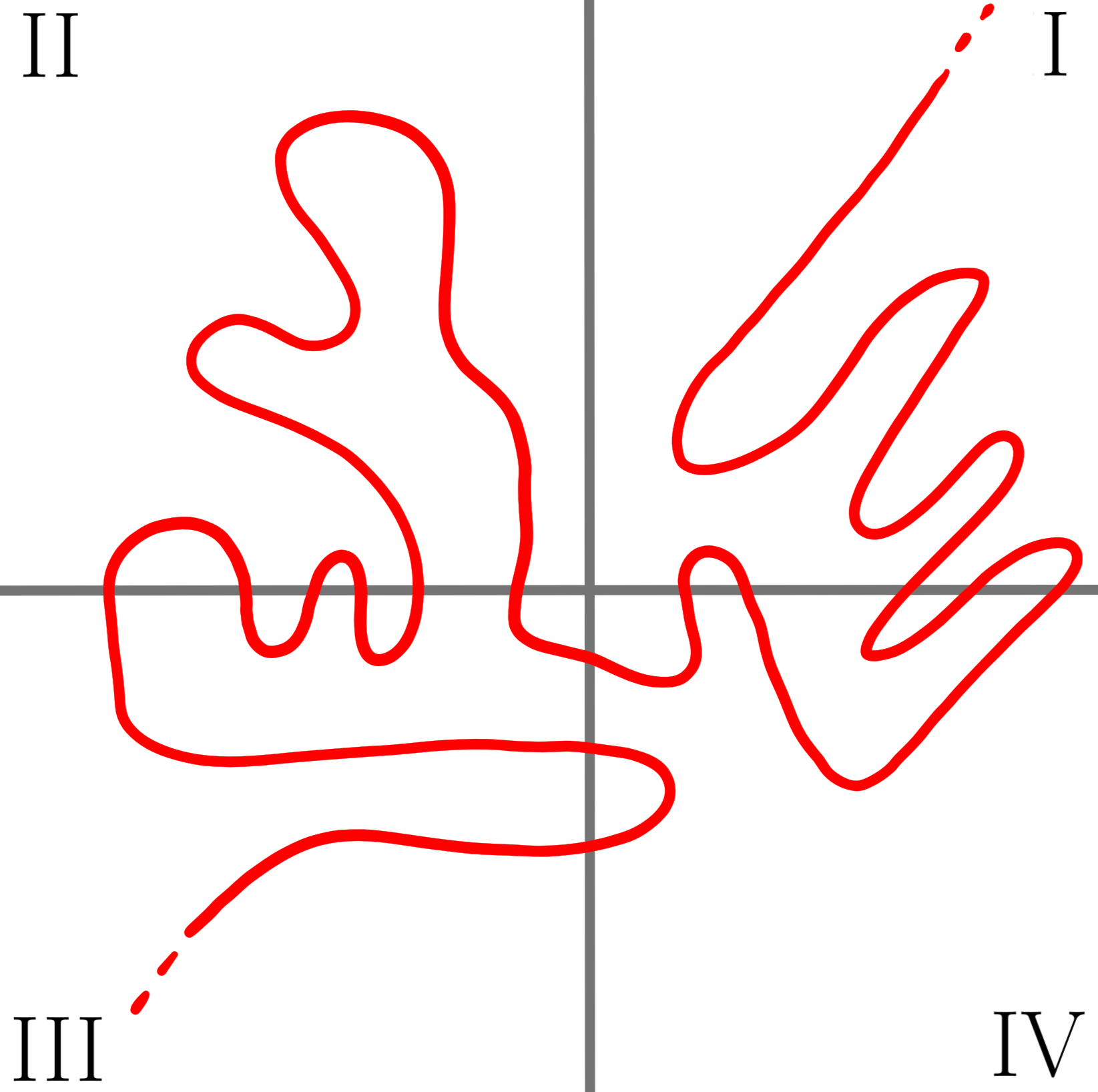}
  \caption{A win for Red}
  \label{Figure.Red-win}
\end{wrapfigure}
And indeed, this informal idea leads directly to a perfectly precise winning criterion, what we shall call the \emph{standard winning condition}, defined as follows. Namely, Red wins a play of infinite Hex if there is a bi-infinite connected chain of red hexagons, a $\mathbb{Z}$-chain of adjacent red hexagons, which on its positive end converges to $(\infty,\infty)$ and on its negative end converges to $(-\infty,-\infty)$. That is, for any choice of center in the hexagonal plane, if one draws vertical and horizontal axis lines at that point, then the positive part of the $\mathbb{Z}$-chain eventually enters quadrant I and stays within it, and the negative part of the $\mathbb{Z}$-chain eventually enters quadrant III and stays within it. And similarly for Blue, but using quadrants II and IV. If there is a $\mathbb{Z}$-chain fulfilling this winning condition, then there must be one that is non-self-crossing, consisting entirely of distinct adjacent tiles of that color, since any tile would be revisited at most finitely often and we can in each case simply omit the intermediate finite steps.

An equivalent formulation of the standard winning condition avoids the need to speak of a center point. Namely, Red wins a play of infinite Hex, if there is a connected $\mathbb{Z}$-chain of adjacent red hexagons such that (i) for every vertical line in the plane, the red chain touches it at most finitely many times, with the positive part ending strictly on the right side and negative part on the left; and (ii) for
\begin{wrapfigure}[12]{l}{.22\textwidth}
\includegraphics[width=.2\textwidth]{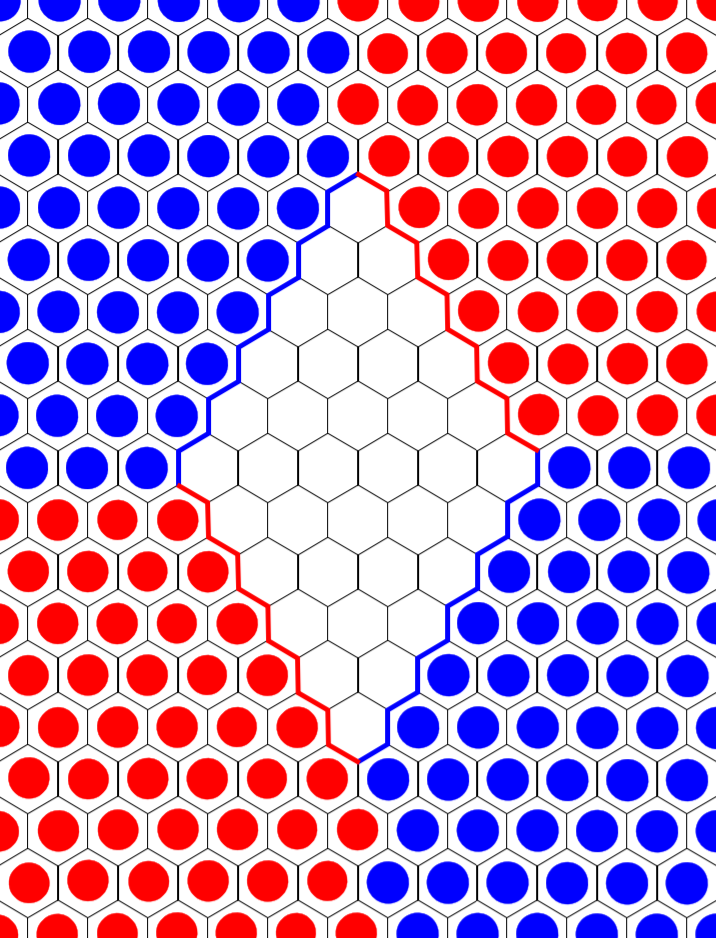}\hfill
\captionsetup{style=leftside,font=footnotesize}
\caption{Finite Hex as infinite Hex}
\label{Figure.Finite-as-infinite}
\end{wrapfigure}
any horizontal line, the $\mathbb{Z}$-chain touches it at most finitely many times, ultimately passing from below, on the negative end of the $\mathbb{Z}$-chain, to above, on the positive end. Succinctly, Red wins if there is a $\mathbb{Z}$-chain of adjacent red hexagons, which crosses every vertical line ultimately from left to right (that is, except for at most finitely many recrossings), and every horizontal line ultimately from below to above. And similarly for Blue in the other diagonal direction.

Any instance of finite Hex on an $n\times n$ rhombus board can be seen as an instance of infinite Hex simply by filling in the four quadrants, as indicated in Figure~\ref{Figure.Finite-as-infinite}. Winning on the finite board is the same as winning on the infinite board and conversely. 

\subsection{Positions with no winner}\label{Section.no-winner}
Having thus specified the winning condition precisely, let us mention several different kinds of troublesome positions that perhaps help to motivate it, which also pose certain challenges for what might be considered promising alternative winning conditions. These examples also show that the infinite Hex analogue of Theorem~\ref{Theorem.Finite-hex} statement (1) does not generally hold.

\enlargethispage{15pt}
\begin{wrapfigure}[9]{r}{.36\textwidth}
\hfill
\includegraphics[width=.33\textwidth]{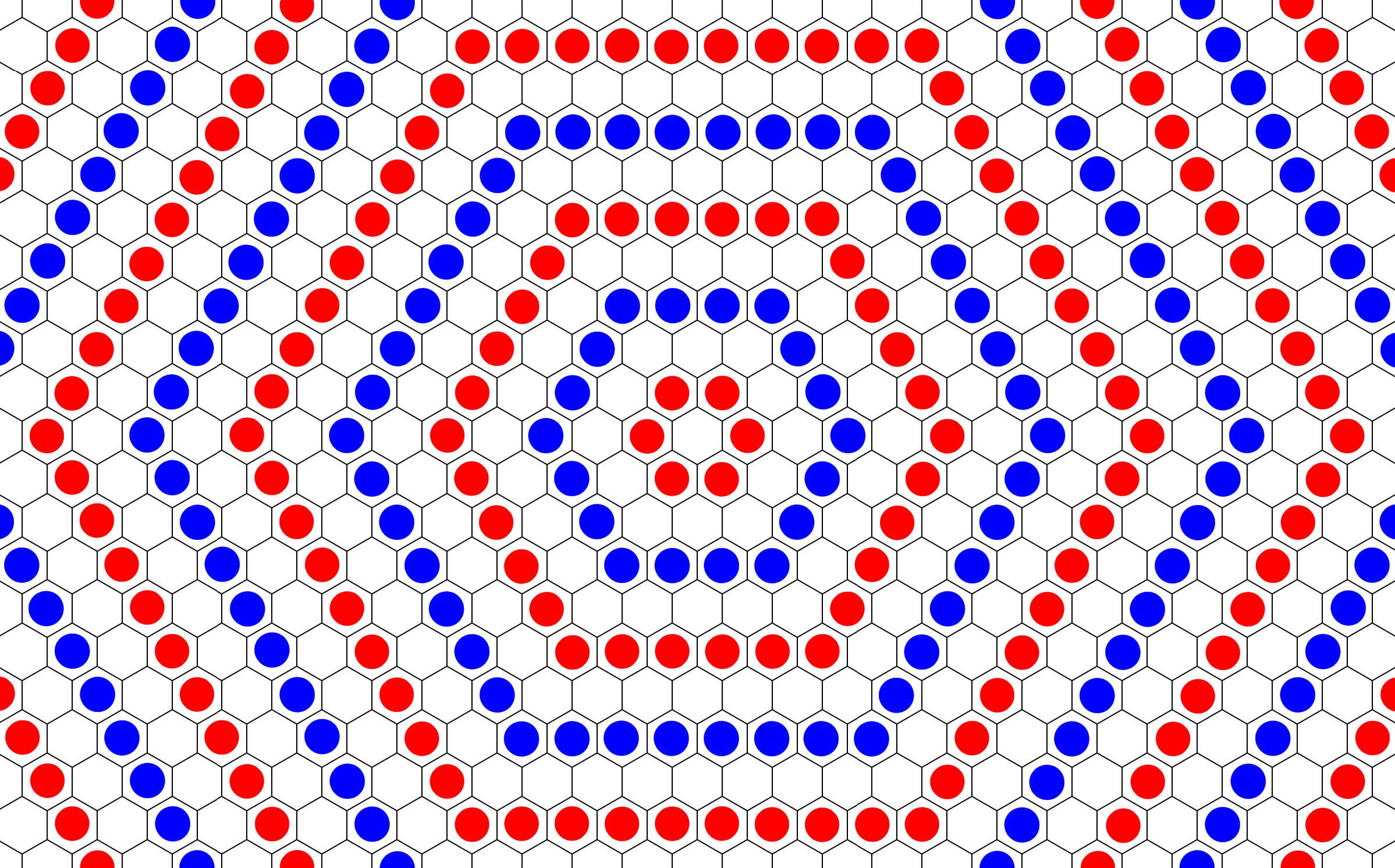}
\captionsetup{style=rightside,font=footnotesize}
\caption{Bullseye}
\label{Figure.Nested-circles}
\end{wrapfigure}
Consider first a position posing an immediate difficulty for the infinitary analogue of Theorem~\ref{Theorem.Finite-hex}, namely, a nested family of red and blue  hexagons, as shown in the bullseye position of  Figure~\ref{Figure.Nested-circles}. For such a game play, since all its monochromatic connected components are finite, it would seem that neither Red nor Blue has succeeded in connecting their ends of the board at infinity, and so we would not want to say that either player has won (and this would remain true no matter how the empty white tiles were completed). And since such a coloring of the plane admits no unbounded $\mathbb{Z}$-chains of either color, neither player fulfills the standard winning condition.\goodbreak

\begin{wrapfigure}{l}{.44\textwidth}
\includegraphics[width=.4\textwidth]{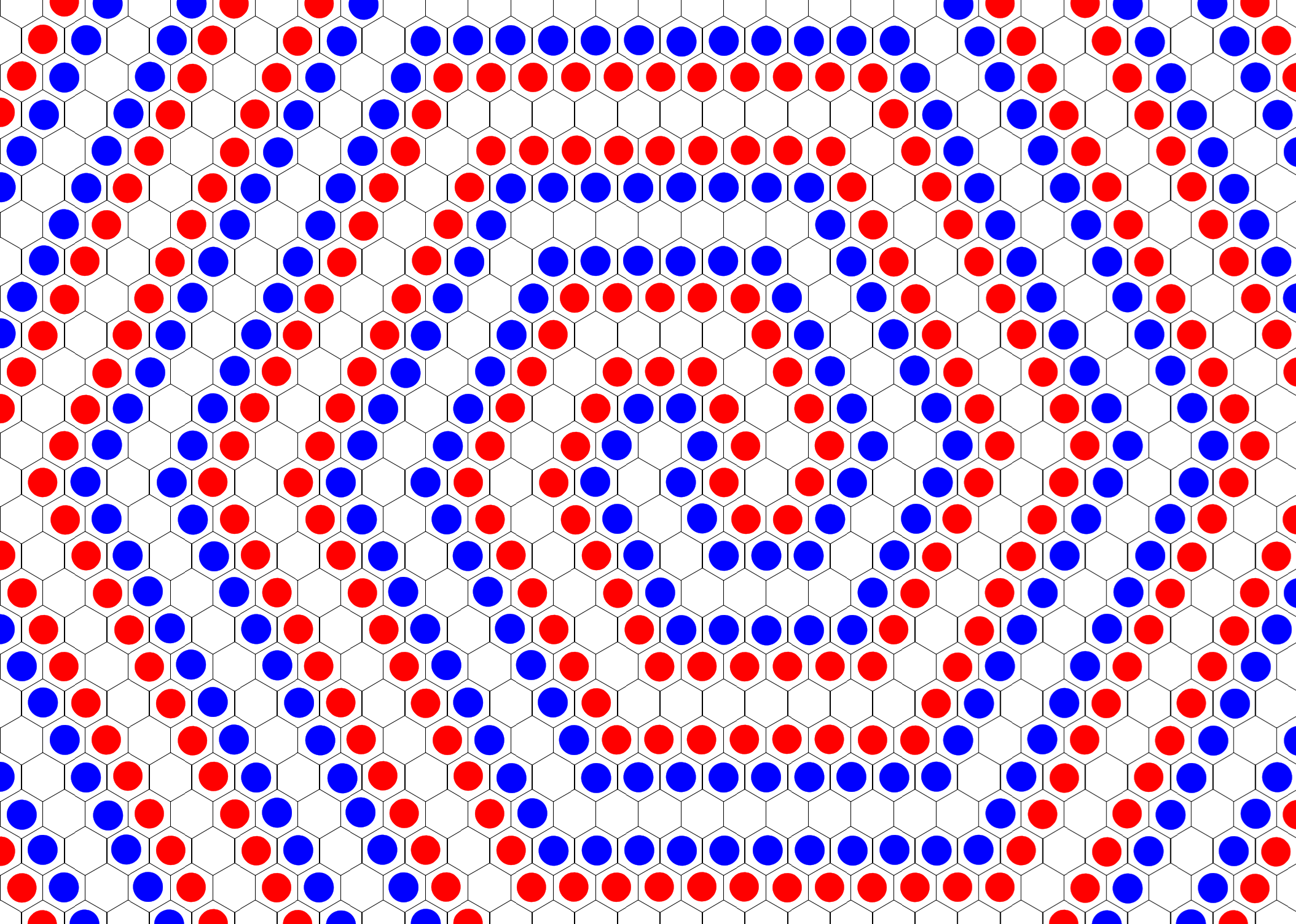}\hfill
\captionsetup{style=leftside,font=footnotesize}
\caption{Spiral paths, no winner}
\label{Figure.Double-spiral}
\end{wrapfigure}
Consider next the double-spiral coloring shown in Figure~\ref{Figure.Double-spiral}. Intuitively, for this pattern neither Red nor Blue seems to have connected their sides of the board at infinity, and so we would seem to want not to award this as a win for either player (and this would remain true no matter how the white tiles were completed). Furthermore, in light of the fundamental symmetry of the position, we certainly would not want to award this as a win for one player rather than the other. This is a case, therefore, for which the standard winning condition seems to get the right result; neither Red nor Blue here meets the standard winning condition---the only available infinite paths for Red and Blue spiral unboundedly in every direction, crossing every horizontal and vertical line infinitely many times.

\begin{wrapfigure}{r}{.33\textwidth}\hfill
\includegraphics[width=.3\textwidth]{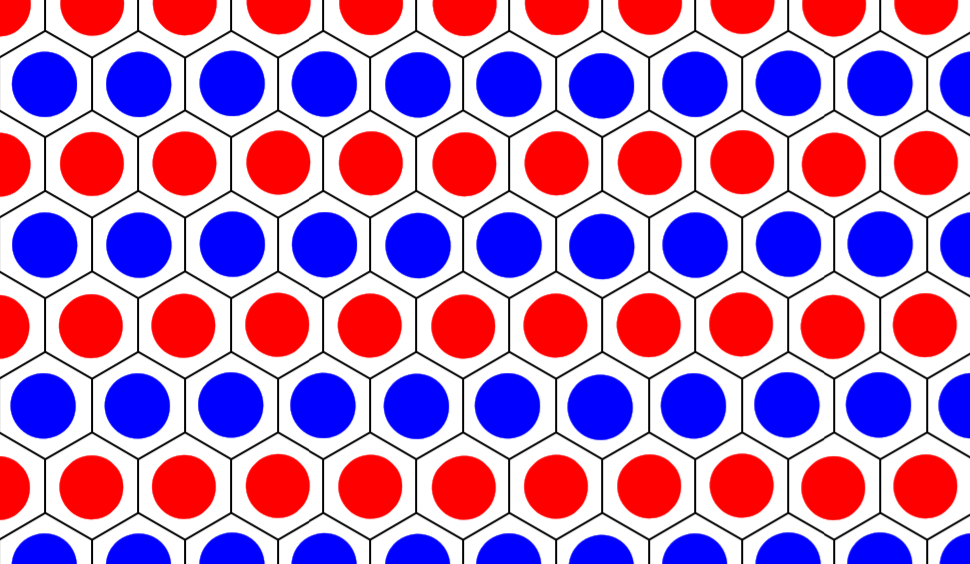}
\captionsetup{style=rightside,font=footnotesize}
\caption{Stripes}
\label{Figure.Horizontal-stripes}
\end{wrapfigure}
Another difficult case is posed simply by a system of red/blue horizontal stripes, as in Figure~\ref{Figure.Horizontal-stripes}, and a similar situation would arise alternatively with (somewhat wiggly) vertical stripes. We should not want to count these as winning for either player, we argue, since the inherent translational symmetry of the position would seem to require us also to count them as winning simultaneously for the other player. But it seems desirable that a winning condition should lead always to at most one winner. For this reason, we are inclined to classify these cases as draws, with no winner, rather than as an outcome with two winners, although there may be a certain irrelevance in that distinction. Meanwhile, the standard winning condition does indeed count this game result as a draw. 

\begin{wrapfigure}{l}{.4\textwidth}
\includegraphics[width=.35\textwidth]{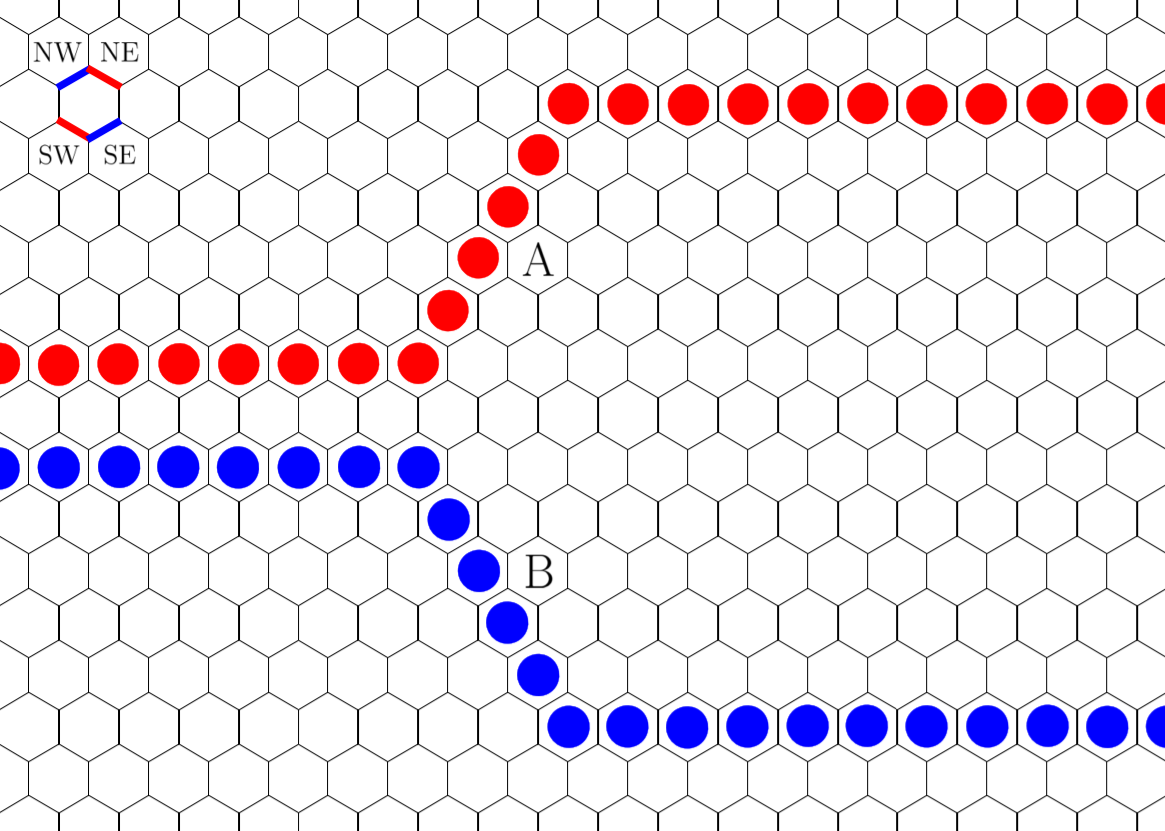}\hfill
\captionsetup{style=leftside,font=footnotesize}
\caption{Bounded paths}
\label{Figure.Tans}
\end{wrapfigure}
Here in Figure~\ref{Figure.Tans} is a more serious kind of case, having to do with the homogeneity of the hexagonal tiles. Specifically, in this position we have an infinite red path resembling a piecewise linear approximation to the graph of the function $y=\arctan x$, proceeding infinitely from the left, and then arcing upward and exiting at right, with horizontal asymptotes at each end. Similarly, a blue path proceeds infinitely from the left, arcing downward and then exiting at right. Notice that the red path proceeds from quadrant III ultimately into quadrant I for the axes using the point $A$ as origin; and the blue path similarly moves from quadrant II to quadrant IV for the axes determined by origin point $B$. Neither of these paths, however, fulfills the standard winning condition, because they are both bounded in the vertical direction, and consequently not winning with respect to all alternative choices of center point, a requirement of the condition. In our view, neither of these paths should entitle its player to win, precisely because in each case the existence of such a curve is compatible with the existence of a similar such curve for the opponent, and we do not want to allow that both players might win a given play. While these curves are not themselves winning, nevertheless they do serve as defensive obstacle configurations preventing the opponent from winning. This kind of situation therefore motivates the standard winning condition by requiring each player to make the crossing of every horizontal and vertical line, not just those associated with a fixed center point. 

One could counterpropose, of course, that a  winning condition might be defined with respect to a given center point, fixed once and for all at the beginning of the game. That is, Red wins a play of the game according to the \emph{fixed-origin winning condition}, if there is a non-self-crossing red $\mathbb{Z}$-chain moving from quadrant III to quadrant I with respect to the axes system determined by the fixed origin point; and similarly for Blue moving from quadrant IV to quadrant II. The argument of Theorem~\ref{Theorem.One-player-wins} will show that for any fixed origin point, it is impossible that both players win a given play with respect to that point, i.e.~the fixed-origin winning condition does not allow both players to win.

The position of Figure~\ref{Figure.Tans} would not necessarily pose a problem for the fixed-origin winning rule---Red would win if the origin was point A, and Blue would win if the origin was point B. But precisely because these outcomes differ, the fixed-origin winning condition violates the inherent homogeneity of tiles on an empty Hex board. Namely, there is something to the idea that at the start of play on an empty board, all tiles look alike---the board after all admits translational symmetries moving any hex to any other hex, while preserving the directions at infinity. It would therefore seem desirable for a winning condition to respect this homogeneity of the hex tiles; it should be translation invariant. But the fixed-origin winning condition does not respect homogeneity and is not translation invariant, as Figure~\ref{Figure.Tans} shows, since Red wins with origin A and Blue wins with origin~B.

A somewhat weaker winning criterion would be the \emph{some-origin} winning condition, by which a player wins a play of the game if they win according to the fixed-origin condition for some choice of origin. This winning condition suffers from the fact that both players can win, as just shown in Figure~\ref{Figure.Tans}.

\enlargethispage{20pt}
Here in Figure~\ref{Figure.Double-prongs} is a more confounding case, since it shows a winding red path that in some sense does seem to move from lower left at infinity to upper right at infinity. This path, however, does not fulfill the standard winning condition, because it visits every horizontal line infinitely often. Nevertheless, in our view
\begin{wrapfigure}{r}{.38\textwidth}
\hfill
\includegraphics[width=.35\textwidth]{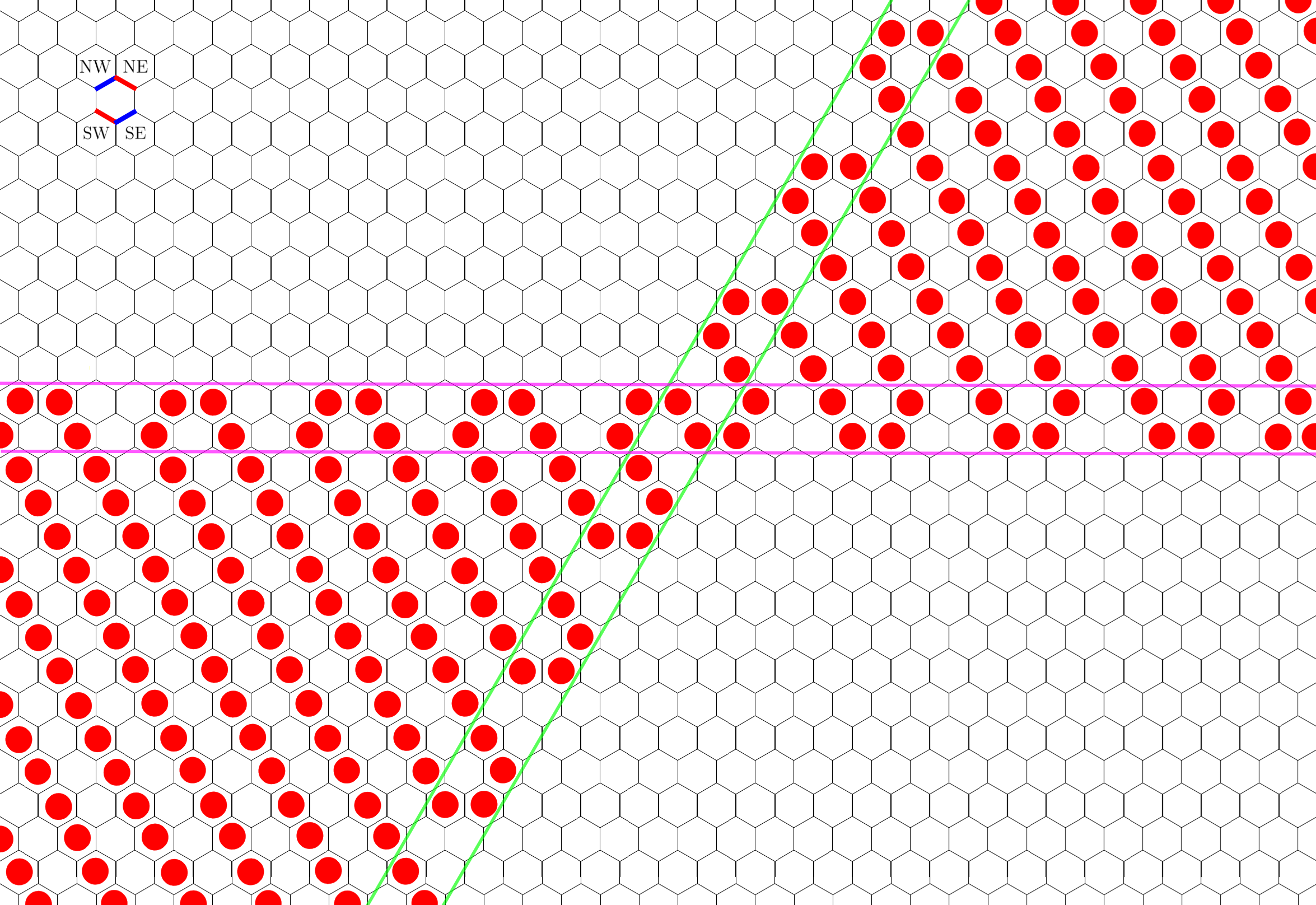}
\captionsetup{style=rightside,font=footnotesize}
\caption{A losing red path}
\label{Figure.Double-prongs}
\end{wrapfigure}
the standard winning condition has got the right result here, in light of the fact that the path keeps revisiting the main horizontal line as it proceeds, never quite definitively departing from the purple center lines, or indeed from any horizontal line. It therefore does not fully succeed in departing ``to infinity'' on either end of its journey, but rather keeps revisiting the center, and we are content to count it as not winning. We do admit, however, that alternative rules might want to allow such a path as this as winning; but ultimately we prefer the natural simplicity of the standard winning condition we have stated.\goodbreak

\begin{wrapfigure}{l}{.44\textwidth}
\includegraphics[width=.4\textwidth]{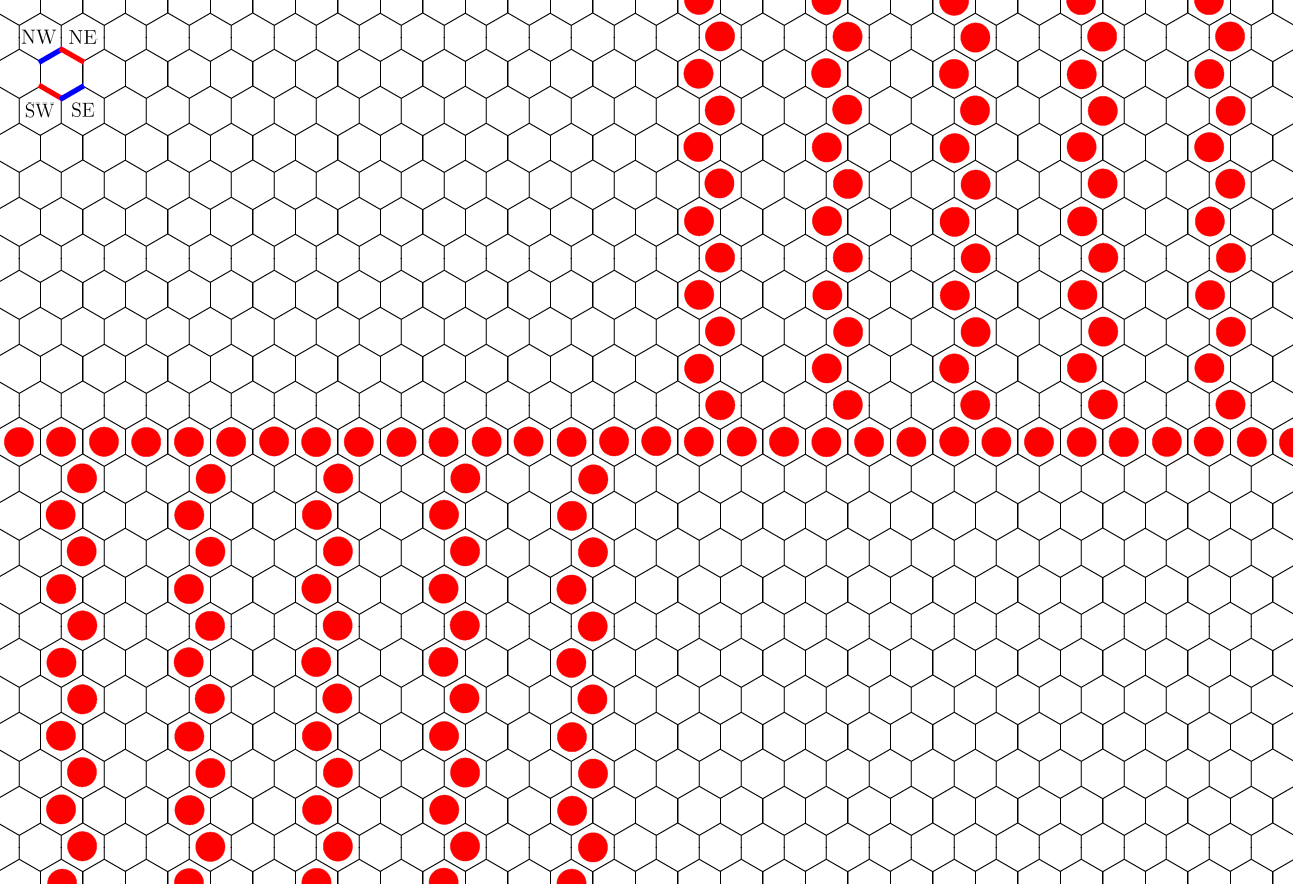}\hfill
\captionsetup{style=leftside,font=footnotesize}
\caption{Infinite double comb}
\label{Figure.Infinite-double-comb}
\end{wrapfigure}
Another example would be the double comb of Figure~\ref{Figure.Infinite-double-comb}, consisting of a bi-infinite horizontal red line, with infinitely many infinite vertical red tines coming off of it. For any center point, Red will win with this shape on any rhombus large enough to encompass enough of the connecting center line. But the position does not fulfill the standard winning condition, since if a  $\mathbb{Z}$-chain of red stones touches the center red line only finitely often, then it must depart on each end onto specific vertical tines, which would make it horizontally bounded.

\subsection{At most one winner in infinite Hex}

The examples we have provided in Section~\ref{Section.no-winner} show senses in which the existence claim of Theorem~\ref{Theorem.Finite-hex} statement (1)---the claim that there is always at least one winner---does not hold for infinite Hex. Nevertheless, we are able to prove the infinitary analogue of the uniqueness part of the claim, that there is always at most one winner. 

\begin{theorem}\label{Theorem.One-player-wins}
In every play of infinite Hex, at most one player wins. Indeed, for any coloring of the infinite hexagonal plane, at most one player fulfills the standard winning condition.
\end{theorem}

\begin{proof}
Suppose toward contradiction that there is a coloring of the infinite Hex playing board in which both players Red and Blue have won. Fix winning $\mathbb{Z}$-chains for each player, and fix any desired center point with corresponding vertical and horizontal axes. Each player's winning $\mathbb{Z}$-chain has touched those axes only finitely many times. Consider a large rhombus centered at the fixed origin point, encompassing within its interior all of those finitely many touching points, including the relevant finite sub-chains connecting them. Since the red $\mathbb{Z}$-chain that crosses the horizontal axis ultimately from below to above crosses the vertical axis ultimately from left to right, it follows that Red has connected the opposite red sides of this rhombus---that is, Red has won this instance of finite Hex; and similarly with Blue. So the coloring on this rhombus-shaped subboard shows both players having won on the finite rhombus, contradicting Theorem~\ref{Theorem.Finite-hex}, statement~(1). So it was not possible after all for them both to have won the infinite game.
\end{proof}

\subsection{Winning on finite boards}\label{Section.Finite-boards}

What is the precise relationship between a play of infinite Hex winning on the infinite board versus winning on arbitrarily large or sufficiently large finite Hex boards? The proof of Theorem~\ref{Theorem.One-player-wins} reveals a connection.

\begin{corollary}\label{Corollary.Win-implies-win-on-sufficently-large-finite}
If a player has won a game of infinite Hex, then for every choice of center, the player has also won finite Hex with this play for all sufficiently large rhombus boards centered at that point.
\end{corollary}

\begin{proof}
Since the choice of center in the proof of Theorem~\ref{Theorem.One-player-wins} was arbitrary, the proof of that theorem shows exactly this.
\end{proof}

\enlargethispage{20pt}
In light of this, let us briefly consider an alternative weaker winning condition. Namely, we define that a play of the game is a \emph{finite-boards win} for a player, if for any choice of center, the player has won on all sufficiently large rhombus boards having that center. Corollary~\ref{Corollary.Win-implies-win-on-sufficently-large-finite} shows that every win in infinite Hex is also a finite-boards win, but Theorem~\ref{Theorem.Winning-on-sufficiently-large-finite-boards} will show that the converse is not true, so the finite-boards win is a strictly weaker winning condition.\goodbreak

It follows easily from Corollary~\ref{Corollary.Win-implies-win-on-sufficently-large-finite} that if a player has achieved a finite-boards win in a play of infinite Hex, then this prevents the opponent from having won the infinite play, because if the opponent had won (according to the standard winning condition), then by Corollary~\ref{Corollary.Win-implies-win-on-sufficently-large-finite} the opponent will have also achieved a finite-boards win, and so the given player could not also have achieved this.  

The finite-boards winning condition offers a few robust features, which might make it attractive as an alternative criterion. First, the arguments we have given show that for any given coloring of the infinite Hex board with red and blue, not both players can achieve a finite-boards win---for any play there is at most one finite-boards winner. Moreover, the finite-boards winning condition is translation invariant, and consequently respects the homogeneity of the tiles on an empty board. Nevertheless we shall have good reason to prefer our standard winning condition, in light of the discussion following Theorem~\ref{Theorem.Winning-on-sufficiently-large-finite-boards}.

Notice that Figure~\ref{Figure.Double-prongs} is not a finite-boards win for Red. To see this, consider a center point above the main red path of Figure~\ref{Figure.Double-prongs}, and observe that for very large rhombuses centered on that point, Red will fail to win on every other size of rhombus, as they grow larger, since when the upper right edge of the rhombus lines up just on the empty line of cells between the red zig-zags, Red will not complete the path. So it will not be true for that position that Red has won on all sufficiently large rhombuses with that center. 

Is the converse of Corollary~\ref{Corollary.Win-implies-win-on-sufficently-large-finite} true? In other words, are the two winning conditions identical? The answer is negative. 

\begin{theorem}\label{Theorem.Winning-on-sufficiently-large-finite-boards}
 There is a play of infinite Hex that is winning with respect to the finite-boards winning condition, but not with respect to the standard winning condition.
\end{theorem}

\begin{wrapfigure}{r}{.44\textwidth}
\vskip-2ex\hfill
\includegraphics[width=.4\textwidth]{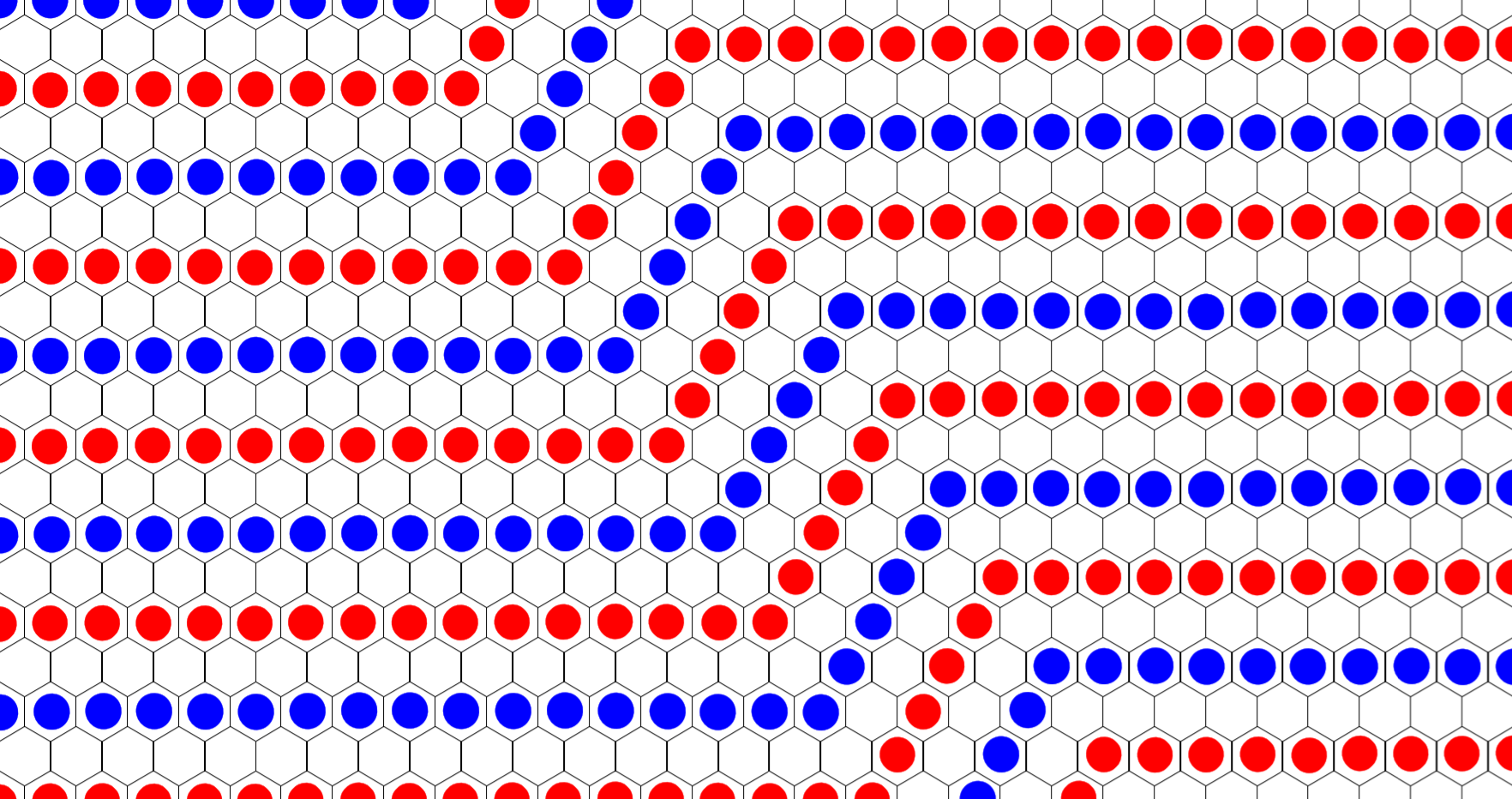}
\captionsetup{style=rightside,font=footnotesize}
\caption{Zen garden position}
\label{Figure.Sufficiently-large-finite-boards}
\end{wrapfigure}
\smallskip\noindent\emph{Proof.} 
Consider the play of infinite Hex resulting in alternating stripes of red and blue in the Zen  garden manner of Figure~\ref{Figure.Sufficiently-large-finite-boards}. For any choice of center, we can make the rhombus big enough so as to encompass the jagged part of the stripes, and in this way, we claim, Red will win, because the red line entering the rhombus just below its center line at left will traverse the rhombus, climb the jagged step, and exit just above the center line of the rhombus at right. So this will be a winning path for that rhombus. For any choice of center, all sufficiently large rhombuses centered at that point will show a win for Red in just this way. But meanwhile, the position overall is not a win for Red in infinite Hex, because no one $\mathbb{Z}$-chain is enough to fulfill the standard winning condition, as each of them is bounded vertically.  

The position shown in Figure~\ref{Figure.Infinite-double-comb} also serves to prove this theorem.
\hfill\fbox{}\medskip\goodbreak 

\enlargethispage{20pt}
Our view of the play shown in Figure~\ref{Figure.Sufficiently-large-finite-boards} is that although Red wins on all sufficiently large boards at any given center, the particular winning paths vary with the choice of center. No one of the red paths should be considered sufficient to justify a win in the infinite game, because each one of them, taken independently, is compatible as in Figure~\ref{Figure.Double-prongs} with the existence of a similar but contrary such blue path; and we do not want to allow two winners for the same play. Ultimately, in our view, each of the red paths of Figure~\ref{Figure.Sufficiently-large-finite-boards} serves a defensive role, sufficient to prevent a Blue win, but insufficient on its own for a Red win.\goodbreak

\subsection{Transfinite play}

In ordinary finite Hex, players naturally proceed with the game if necessary until every hex tile is occupied by a stone. Meanwhile, with infinite Hex, there are infinite plays of the game, plays for which infinitely many stones have been placed, yet huge regions of the board remain nevertheless untouched---the game may remain completely unsettled. For example, perhaps Red plays a half-chain reaching from the origin out to infinity in quadrant I, while Blue plays a similar half-chain in quadrant II, leaving the rest of the board completely empty even after infinitely many stones are placed. In such a case, it would seem natural simply to allow the players to continue playing---let the game continue into  transfinite ordinal time. That is, if the moves at all finite stages have been completed and there remain empty hex tiles, then the players may continue placing stones at infinite ordinal stages $\omega$, $\omega+1$, $\omega+2$, and so on. At some countable ordinal stage of play, the board will become completely filled and the game will be over. 

In \emph{standard play}, the game of infinite Hex is concluded as a win, loss, or draw after $\omega$ many moves, whether or not the board is filled with stones. With \emph{transfinite play}, in contrast, game play proceeds if necessary into transfinite ordinal time. In this latter case, one must specify as a matter for the rules whose turn it is at the limit ordinal stages of play---who gets to play at stage $\omega$? At $\omega\cdot 2$? By default we might specify that Red always plays first and also first at every limit stage; but other rules are also sensible. Perhaps one might want to allow Blue to play at limit stages in compensation for Red's advantage of the initial first move. Or perhaps one might say that Red plays first and at all compound limits, while Blue plays at simple limits. But in fact our title result (Theorem~\ref{Theorem.main}) shows that when play begins on an empty board, then both players have drawing strategies, regardless of who plays first and who plays at limit stages.

\subsection{Complexity of the standard winning condition}

What is the complexity of the standard winning condition in infinite Hex? It seems natural to try to place the game into the complexity hierarchies of descriptive set theory. We had been unsure of the answer exactly and inquired about this on MathOverflow.

\begin{question}[\cite{Hamkins2022.MO435565:What-is-the-complexity-of-the-winning-condition-in-infinite-hex}]\label{Question.Complexity-winning-condition}
What is the descriptive-set-theoretic complexity of the game of infinite Hex? 
\end{question}

Specifically, what is the complexity of the set of partial colorings of the infinite Hex board fulfilling the standard winning condition for Red? There is a clear upper bound of analytic complexity $\Sigma^1_1$, of course, because a given coloring exhibits a Red win if and only if there is a connected $\mathbb{Z}$-chain of red stones fulfilling the convergence requirements of the standard winning condition. This is an existential quantifier to assert the existence of the red $\mathbb{Z}$-chain, and to assert that a given chain achieves the winning convergence property is arithmetic, making the complexity of the winning condition $\Sigma^1_1$ altogether. 

One might have hoped to prove that this is optimal by showing that the winning condition is complete for $\Sigma^1_1$. Since the ill-foundedness of countable trees is known to be $\Sigma^1_1$ complete, it would suffice to embed trees into the Hex positions in such a way that they are winning if and only if they are ill-founded. But unfortunately, this idea doesn't pan out---there just isn't enough room in the plane to embed the trees successfully. The main difficulty is that one must embed trees that are infinitely branching, and the infinitely branching nodes will be modeled with a long Hex line having infinitely many offshoots; but one must ensure that the branching line itself is not a winning chain, and so it must in effect return infinitely often to a given horizontal or vertical, and this means that there  won't be room to handle the higher branching nodes of the various offshoots---it just doesn't fit.

Is infinite Hex a Borel game? We conjecture that it is. To show this, it would suffice to provide a real existential reason also for \emph{failures} of winning. For example, perhaps Red fails to win if and only if there is a certain kind of obstacle in the blue and uncolored stones. This would show that the failure to win is also $\Sigma^1_1$, making the winning condition $\Delta^1_1$ and hence Borel. 

In our view, the complexity of the game is a vital consideration---especially the distinction between $\Sigma^1_1$ and Borel---since this has fundamental strategic consequences in the game-theoretic analysis. If the winning condition is Borel, after all, then the game proceeding from a given position will be determined (either one of the players will have a winning strategy or both have drawing strategies) as a consequence of Borel determinacy, which is provable in Zermelo-Fraenkel set theory ZFC. But if the winning condition is truly $\Sigma^1_1$, in contrast, then it wouldn't be as clear that all such games are determined, since that level of determinacy is not provable in ZFC (although it is provable from large cardinals). Can we prove that all positions in infinite Hex are determined, having either a winning strategy for one of the players or drawing strategies for both players?

Is infinite Hex an arithmetic game? We had speculated that it may be, and for this reason we are now pleased to see that Ilkka Törmä \cite{Torma2023.MO450529:What-is-complexity-of-winning-condition-of-infinite-hex} has just recently announced a proof on MathOverflow of exacty that: the standard winning condition of infinite Hex is arithmetic, with complexity at most $\Sigma^0_7$. This seems to settle question \ref{Question.Complexity-winning-condition} and the other questions we have asked in this section.


\section{Infinite Hex is a draw}

We now come to the title result, namely, infinite Hex is a draw.

\begin{theorem}\label{Theorem.main}
Infinite Hex is a draw---neither player has a winning strategy; both players have drawing strategies.
\end{theorem}

\begin{wrapfigure}{r}{.5\textwidth}    
\vskip-2ex\hfill
\includegraphics[width=.45\textwidth]{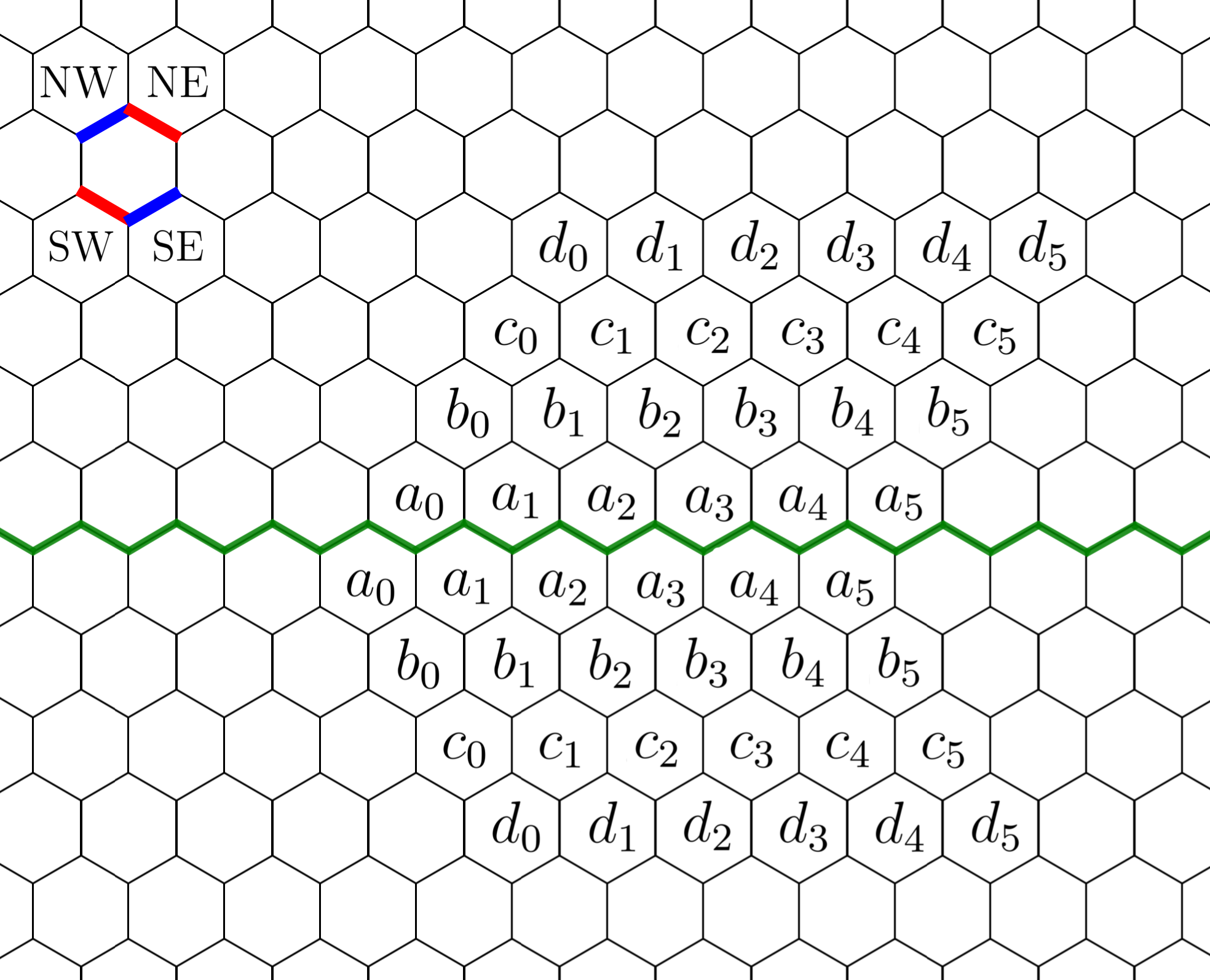}
\captionsetup{style=rightside,font=footnotesize}
\caption{Mirrored pairing of tiles}
\label{Figure.mirroring_empty}
\end{wrapfigure}
\smallskip\noindent\emph{Proof.} 
The strategy-stealing argument of Nash in Theorem~\ref{Theorem.Finite-hex} works equally well in infinite Hex, and so we know that the second player cannot have a winning strategy in infinite Hex. To prove the theorem, therefore, it will suffice for us to show that the second player has a drawing strategy. Let us assume that Red plays first and Blue second. We shall describe a certain mirror-symmetric copying strategy for Blue, which we shall argue is a draw-or-better strategy for Blue.\goodbreak 

At the outset of the game, Blue should fix an arbitrary horizontal line on the board, as shown in green in Figure~\ref{Figure.mirroring_empty}, and consider the induced pairing of hexagonal tiles above and below that line defined by reflection-and-half-tile-shift with respect to it, with the two tiles marked $a_0$ matching, and similarly with $a_1$, $b_0$, and so on. That is, every hexagon below the line is paired with its reflected copy above, shifted half a tile to the right, in the pattern indicated in the figure.

The strategy we have in mind for Blue is simply to play always in the counterpart tile of any move by Red according to this pairing of the  tiles. If Red plays in a tile below the center green horizontal, then Blue will play in the corresponding partner
\begin{wrapfigure}{l}{.5\textwidth}\vskip-1ex    
\includegraphics[width=.45\textwidth]{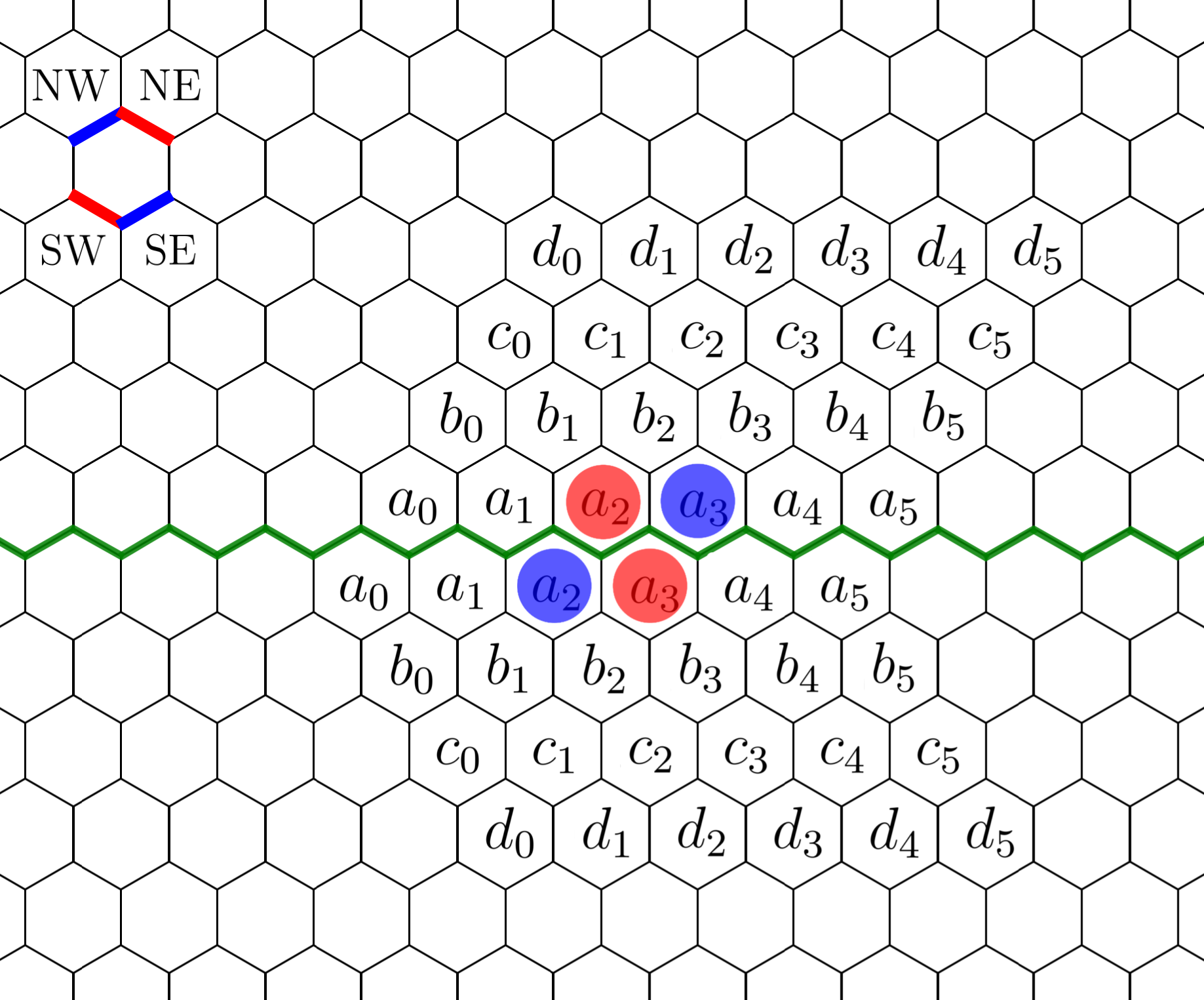}\hfill
\captionsetup{style=leftside,font=footnotesize}
\caption{Red crosses Blue's horizontal}
\label{Figure.mirroring-1-intersection}
\end{wrapfigure}
tile above it, and vice versa. This pairing of tiles is almost but not quite a geometric reflection---the half shift means that the correspondence does not always preserve adjacency of tiles, and indeed the mirroring process can cause tears in connected paths---notice that red tiles $a_2$ and $a_3$ in Figure~\ref{Figure.mirroring-1-intersection} are adjacent, but they are copied by Blue to blue tiles $a_2$ and $a_3$, which are not adjacent. So it isn't quite true that connected red chains are necessarily copied to connected blue chains by this mirroring strategy, and that is not how our argument proceeds. Rather, we shall make a subtler topological 
\enlargethispage{10pt}%
observation, relying on a particular consequence of the counterpart arrangement for the manner in which Red might traverse the center line. Namely, when Blue follows the mirroring strategy, the only way for Red to place two contiguous stones traversing the green center line is in a downward-to-the-right sloping direction, as shown in Figure~\ref{Figure.mirroring-1-intersection}. This simple geometric fact will have an important topological consequence later in the main argument.\goodbreak

\begin{wrapfigure}{r}{.5\textwidth}
\vskip-2ex\hfill
  \includegraphics[width=.45\textwidth]{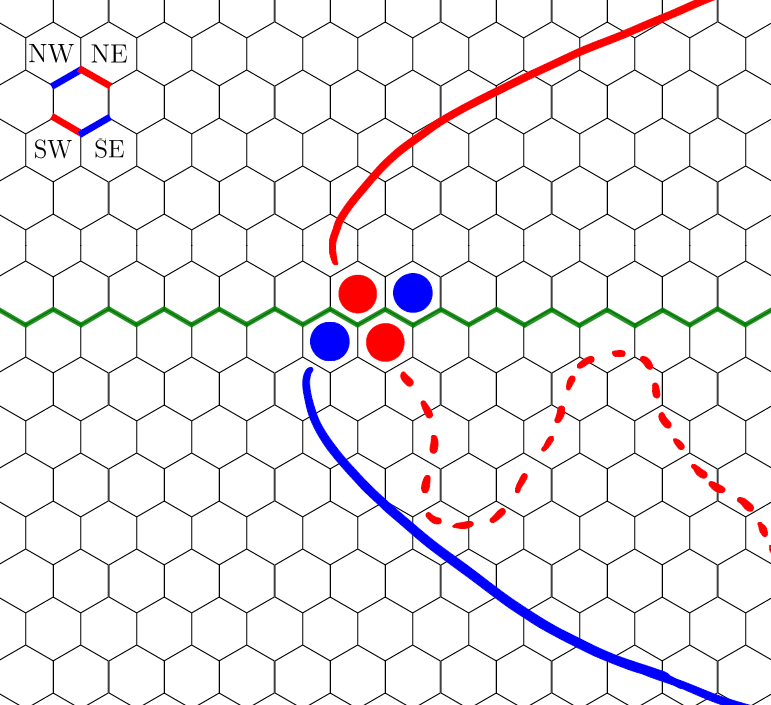}
  \captionsetup{style=rightside,font=footnotesize}
\caption{Red traversing the boundary}
  \label{Figure.Red-traversing-boundary}
\end{wrapfigure}
To prove the theorem we claim that the mirroring strategy is a draw-or-better strategy for Blue. To see this, suppose we have a play of the game in which Red has won, whilst Blue has followed the symmetric mirroring strategy we have described. Descending from infinity at the above right, the winning red $\mathbb{Z}$-chain must eventually traverse the green boundary at some point. We may assume without loss of generality that this red $\mathbb{Z}$-chain is non-self-crossing. Consider the place where the positive branch of the winning red $\mathbb{Z}$-chain crosses the green boundary for the first time, coming from infinity at the above right, say at the point indicated in Figure~\ref{Figure.Red-traversing-boundary}.\goodbreak

The key observation is that because Blue has been following the mirroring strategy, the entire play above the green center line is reflected (and slightly shifted) below, with the colors swapped. In particular, the red path arriving from infinity at upper right will be reflected to a blue path arriving from infinity at lower right. Because of the particular nature and geometry of the crossing, it follows, we claim, that the rest of the red $\mathbb{Z}$-chain after the traverse of the green boundary will be trapped in the right half plane---it will be trapped in the region to the right of the blue path in the lower half plane and the region to the right of the red path in the upper half plane (since we assumed that the chain is not self-crossing). Therefore, the entire red $\mathbb{Z}$-chain will be bounded to a right half plane, and so it cannot be part of a winning red $\mathbb{Z}$-chain, which would be required to depart to infinity at lower left. 

In summary, no coloring of tiles on the board that respects the Blue mirroring strategy can exhibit a winning red path, and so this mirroring strategy will ensure a draw-or-better for Blue, as we claimed. \hfill\fbox{}\medskip\goodbreak 

The argument works whether one uses ordinary infinite play of length $\omega$ or allows transfinite play beyond $\omega$, regardless of who has the right to play first at limits---if Red plays first at a limit, then Blue can play the mirroring move; if Blue plays first at a limit, then Blue can invent an imaginary Red move to copy. In this way, the mirroring strategy for Blue can be played just as well for transfinite play, and the proof of Theorem~\ref{Theorem.main} shows that there can be no position compatible with 
that strategy that has a winning red $\mathbb{Z}$-chain. The proof also works for the fixed-origin weak winning condition, since Blue could simply pick the green center line to contain that origin point. Indeed, the mirroring strategy for Blue results in a play for which Red does not win with respect to any origin at all, since if the upper right red path is in quadrant I with respect to that origin, then the lower blue path will force the other side of the red chain into quadrant~IV.\goodbreak

Let us also observe that the Blue mirroring strategy does not necessarily force a draw, but rather a draw-or-better for Blue. In the game play shown in Figure~\ref{Figure.mirroring-blue-win}, for example, Blue has followed the mirroring strategy, and the result is a win for Blue,
\begin{wrapfigure}{r}{.44\textwidth}
\hfill
\includegraphics[width=.4\textwidth]{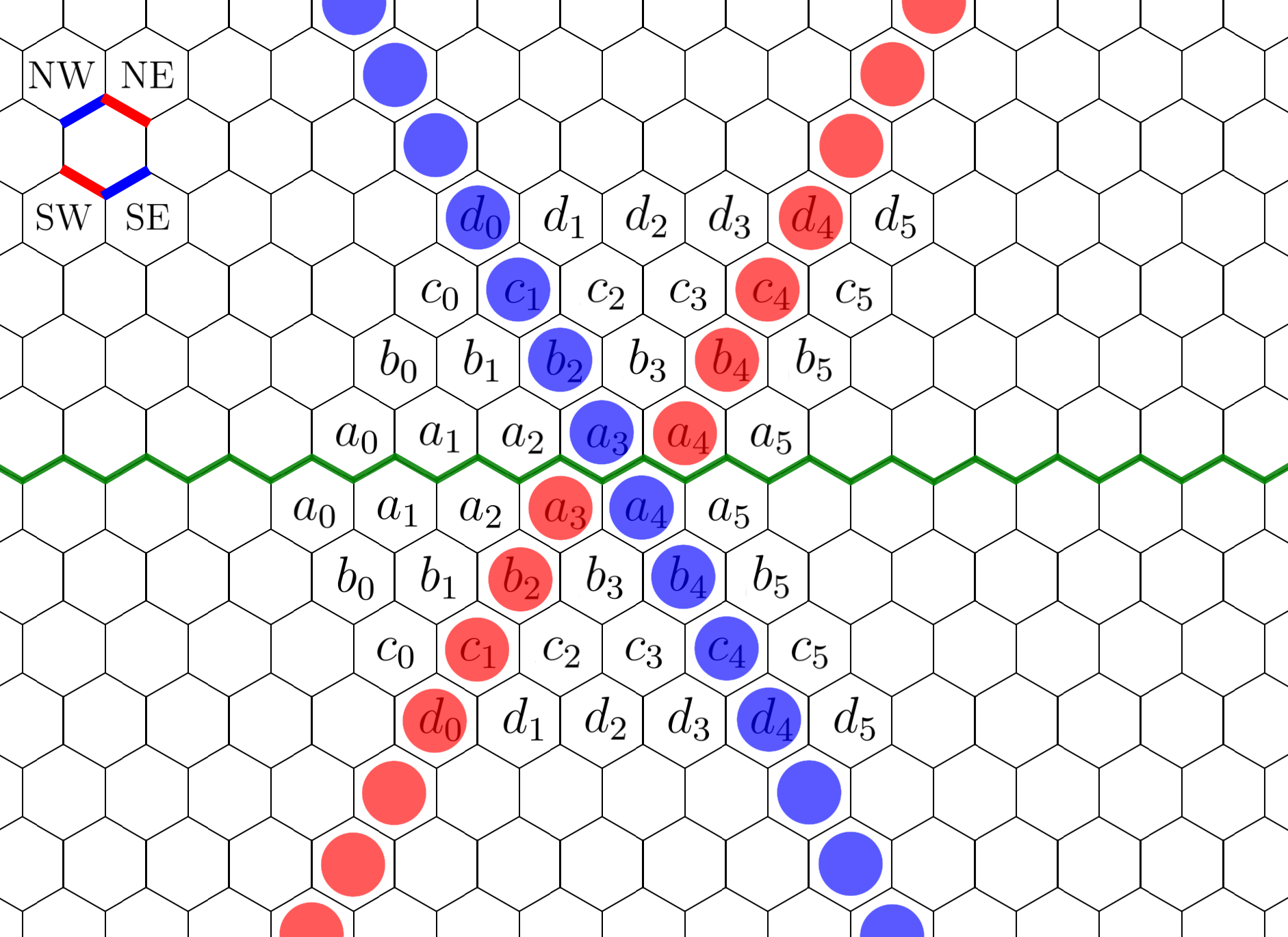}
\captionsetup{style=leftside,font=footnotesize}
\caption{A Blue mirroring win}
\label{Figure.mirroring-blue-win}
\end{wrapfigure}
not a draw. It would admittedly be fairly silly for Red to play this way, however, once the crossing point stones are placed, and we mention this example merely to show that the mirroring strategy by itself does not force a draw, but draw-or-better. Meanwhile, the fact that Blue has a drawing strategy shows by the strategy-stealing argument that Red also has a drawing strategy, and so with optimal play for both players, infinite Hex is a draw.

As a historical remark, we can see that the mirroring strategy generalizes the winning strategy for finite Hex on asymmetric boards---see \cite[\textsection 8]{gardner88}---Theorem~\ref{Theorem.main} can also be proved via a contradiction which reduces a Red winning $\mathbb{Z}$-chain to a winning finite chain for the disadvantaged player in asymmetric Hex.

\section{The finite advantage conjecture}

We should like next to consider how robust is the drawing-strategy phenomenon that we have identified in Theorem~\ref{Theorem.main}. If the first player gets the advantage of an extra stone on the first move, are the scales tipped in their favor? Or is the game still a draw? What if the first player is given the advantage of finitely many extra stones? Is infinite Hex still a draw, if one should start from an arbitrary finite position? We conjecture indeed that there is no such finite advantage.

\begin{conjecture}[Finite advantage conjecture]\label{Conjecture.Finite-advantage}
 Infinite Hex remains a draw when starting from a board position with finitely many stones already placed.
\end{conjecture}\goodbreak

\begin{wrapfigure}{r}{.45\textwidth}
\vskip-1.5ex\hfill
\includegraphics[width=.4\textwidth]{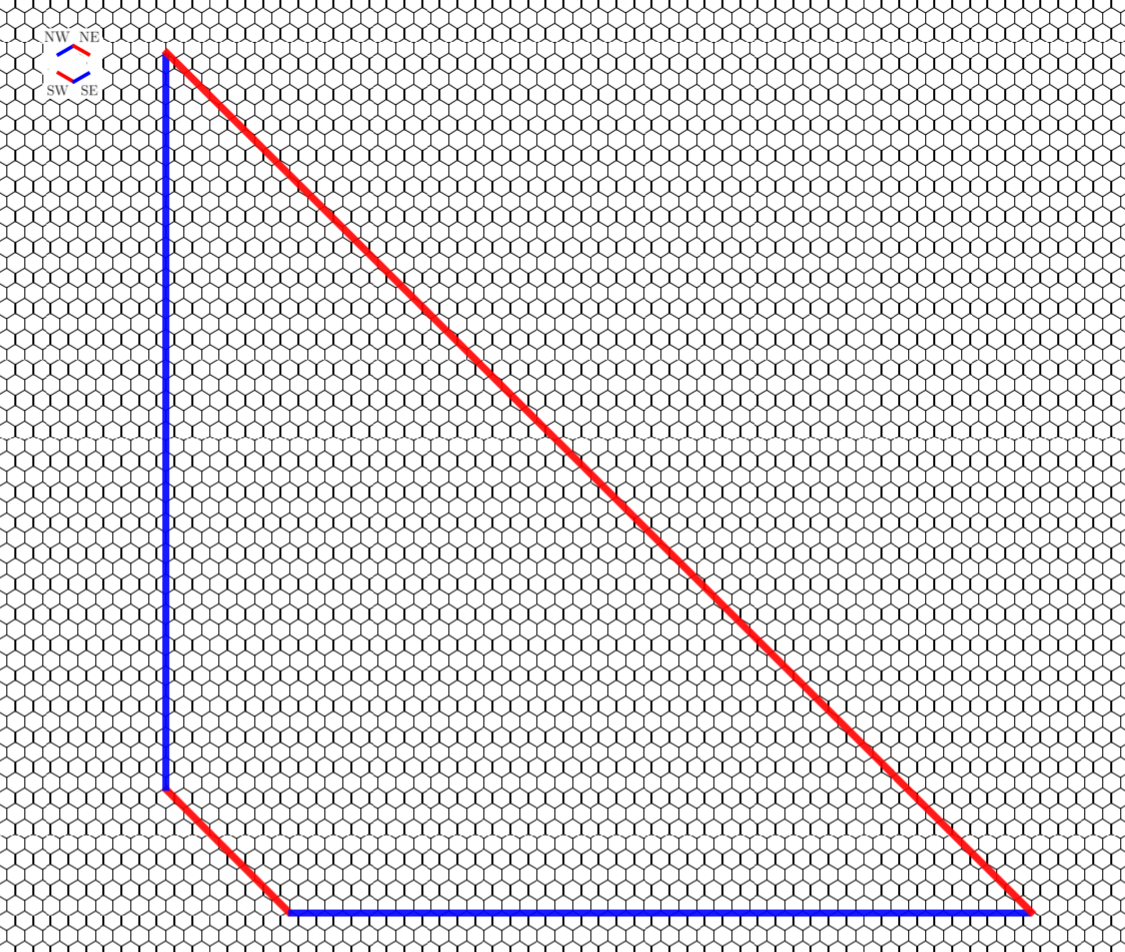}
\captionsetup{style=rightside,font=footnotesize} \caption{Trapezoid board}
\label{Figure.trapezoidal-board}
\end{wrapfigure}
\subsection{Winning on trapezoids}\enlargethispage{40pt}
The conjecture is open in the general case, but we shall prove that it holds under the assumption that Red can win certain configurations of finite Hex, although we are unsure whether this additional assumption is true. We shall also prove the conjecture holds for a certain strengthening of the winning condition.

To begin, consider a large isosceles trapezoid Hex board, as shown in Figure~\ref{Figure.trapezoidal-board}. Red aims to connect the two parallel sides of the trapezoid and Blue to connect the two orthogonal sides. The trapezoid is obtained by truncating the right-angle corner of an isosceles right triangle, and is determined by two lengths, the length of the shorter red side and the length of the blue segments.\goodbreak

\begin{theorem}\label{Theorem.If-trapezoids-then-draw}
 If Blue has a first-player winning strategy for arbitrarily large trapezoids as above, then there is no finite advantage in infinite Hex. That is, on an infinite Hex board with finitely many stones already placed, both players have drawing strategies.
\end{theorem}

\smallskip\noindent\emph{Proof.} 
By ``arbitrarily large,'' what we mean is that Blue, playing first, can win instances of finite Hex played on such trapezoids with the truncated length as large as desired, and the orthogonal sides sufficiently long that Blue has a winning strategy on that finite board. In other words, we assume that for any desired length of the shorter red side of the trapezoid, we can find such a trapezoid that is also sufficiently thick, with the other red side perhaps very far away, such that Blue moving first has a winning strategy to join the blue sides on that trapezoid. (Clearly Red will win if the trapezoid is too thin, with the red sides too close---we care only about sufficiently thick trapezoids that Blue might win.) 

\begin{wrapfigure}{r}{.45\textwidth}
\vskip-1ex\hfill
\includegraphics[width=.4\textwidth]{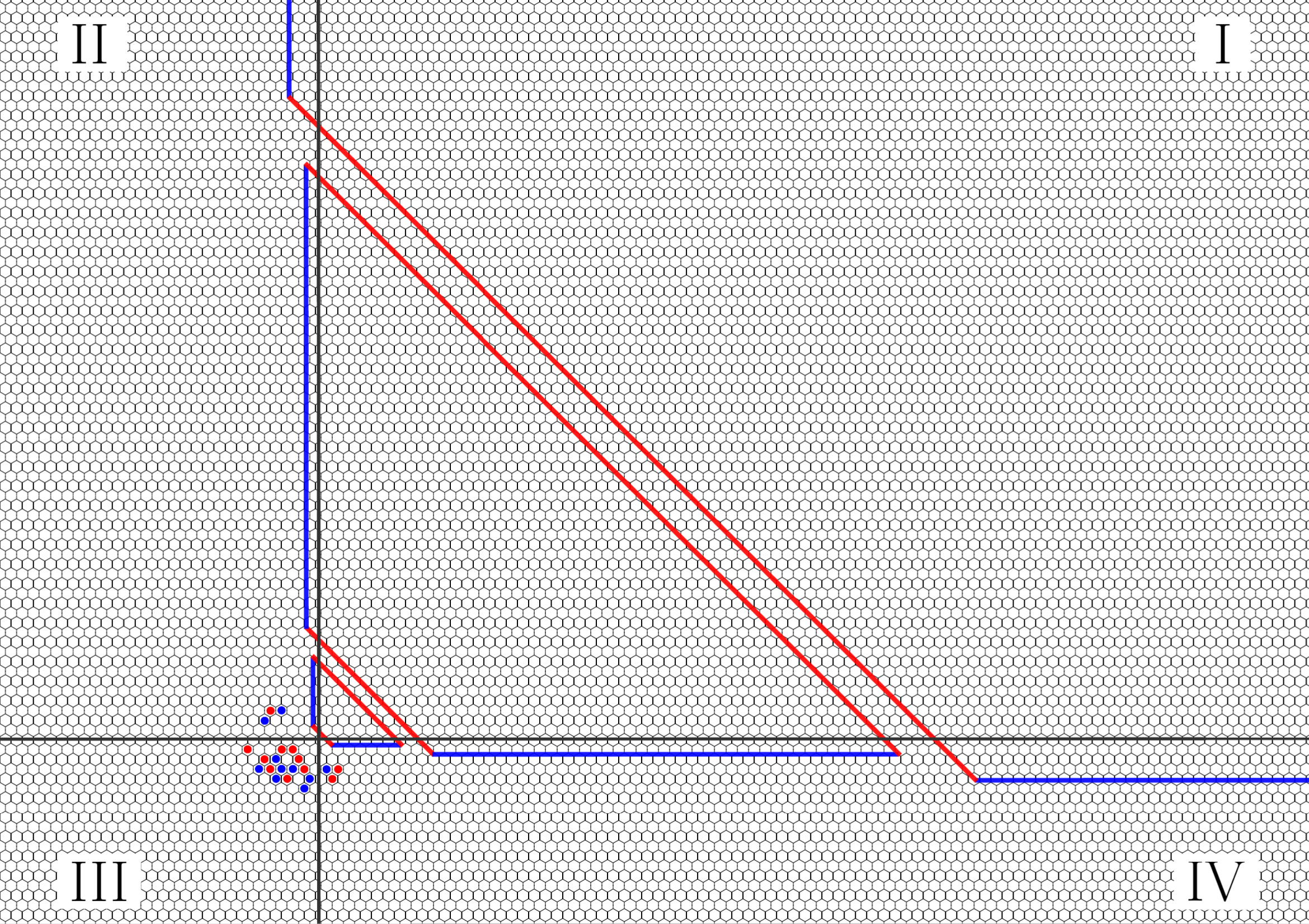}
\captionsetup{style=rightside,font=footnotesize}
\caption{Blue aims to block Red by winning on trapezoids}
\label{Figure.Blue-wins-each-trapezoid}
\end{wrapfigure}
It will suffice by symmetry to prove that Blue has a drawing strategy in infinite Hex proceeding from a position in which finitely many stones have already been placed. Recall that, in order to win a play of the game, Red must build a $\mathbb{Z}$-chain of adjacent red stones, such that for any choice of center, the positive part of the $\mathbb{Z}$-chain eventually enters quadrant I as determined by every such origin point.

Fix a center, and we shall mount a blocking strategy for Blue by dividing the upper right quadrant above the finite existing play into infinitely many disjoint winning-for-Blue trapezoids, as indicated in Figure~\ref{Figure.Blue-wins-each-trapezoid}. Given such an arrangement of trapezoids, the main line of play for Blue is to select the first available empty trapezoid as the currently ``active'' one and make an initial winning move on it. Blue will then continue to play on this active trapezoid (even if Red plays elsewhere), using the winning strategy so as to ensure a Blue connection of the two sides. When this active trapezoid is won, Blue deactivates it and proceeds to activate another, larger empty trapezoid and play so as to win on that one, and so on in this same manner. At any stage of play, Blue is building a winning connection in the currently active trapezoid, ignoring Red moves outside of that trapezoid, and then afterward proceeding to activate another larger trapezoid. Ultimately, in this way Blue will build infinitely many obstacle paths, connecting the two sides of infinitely many of the trapezoids. These blue obstacles will prevent Red from winning the infinite game, because any red $\mathbb{Z}$-chain going to infinity would have to zig-zag in and out of quadrant I in order to get around them. Thus, this is a drawing strategy for Blue in the infinite Hex game.
\hfill\fbox{}\medskip\goodbreak 

The argument succeeds also with respect to transfinite play, since Blue will have built infinitely many obstacle paths already by stage $\omega$, and so further Red moves even in transfinite time would be of no help as the blue obstacles must still be circumvented. Moreover, the theorem also holds with respect to the fixed-origin winning condition, since Blue can set up the trapezoids for any desired center point. A slight variation of the argument works with respect to the some-origin winning condition, since Blue can make the trapezoids increasingly wide, as hinted at in Figure~\ref{Figure.Blue-wins-each-trapezoid}, and in this manner, they would eventually form obstacles for any given center point, not just the center originally chosen.

\begin{corollary}\label{Corollary.Trapezoid-Blue-draw}
If Blue has a first-player winning strategy for arbitrarily large trapezoids, then for any position in infinite Hex leaving either some upper right quadrant or lower left quadrant empty, Blue will have a draw-or-better strategy.
\end{corollary}

\begin{proof}
This is what the argument of Theorem~\ref{Theorem.If-trapezoids-then-draw} shows, since all that was needed was some point whose upper right quadrant was free, in order to place the trapezoids. A similar argument would work in the case that a lower left quadrant was available.
\end{proof}

\begin{wrapfigure}{r}{.45\textwidth}\vskip-1.5ex
\hfill
  \includegraphics[width=.4\textwidth]{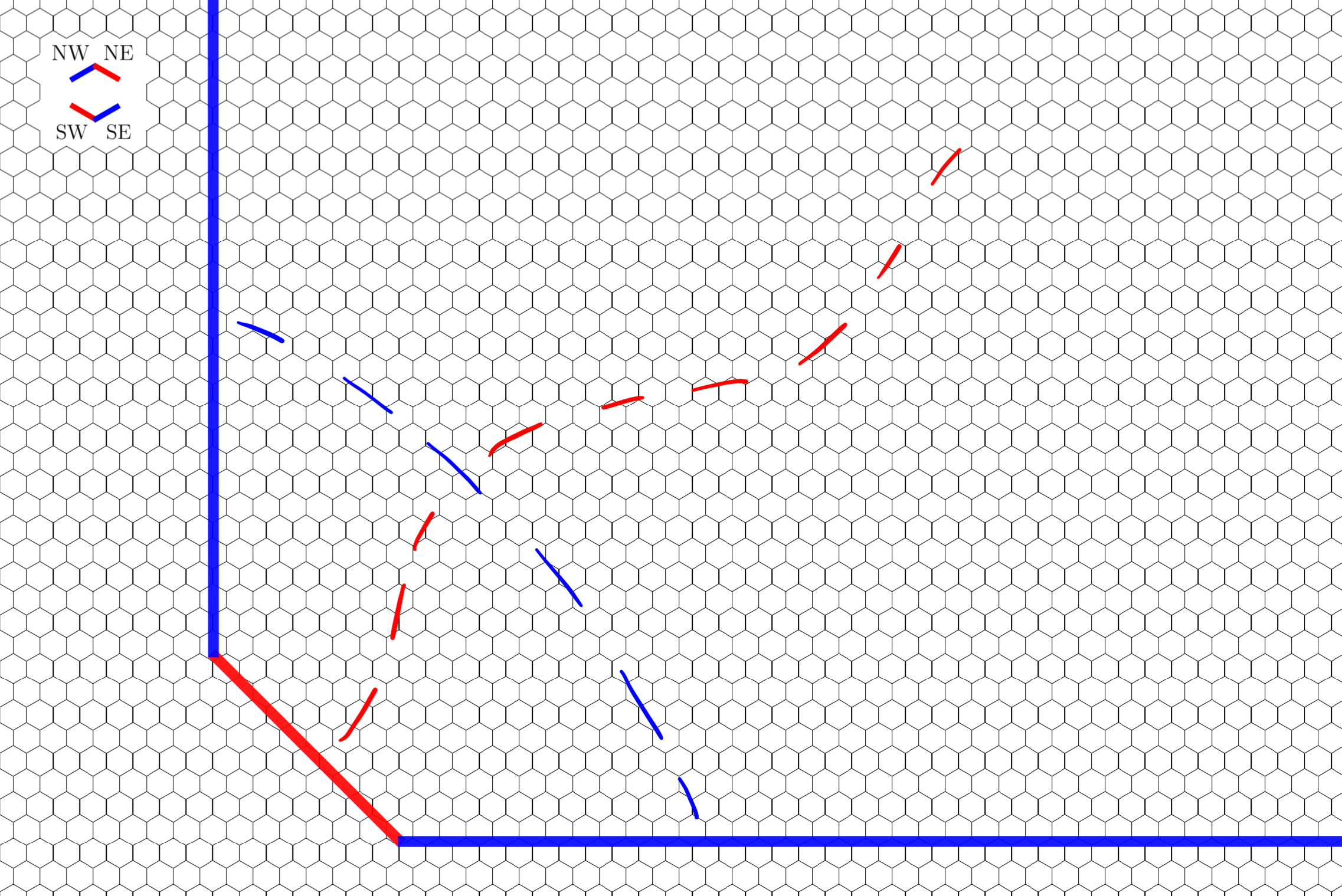}
  \captionsetup{style=rightside,font=footnotesize}   
  \caption{Can Blue connect the sides?}
  \label{Figure.truncated quadrant}
\end{wrapfigure}
It is natural to inquire about an infinitary version of the trapezoid assumption, namely, whether Blue can connect the two sides on every truncated quadrant, as shown in Figure~\ref{Figure.truncated quadrant}, no matter how large the red truncation in the lower left corner. Blue aims to connect the two sides, while Red aims to prevent this, by connecting the corner truncation boundary to infinity. This truncated quadrant game exhibits several attractive game-theoretic features, such as the fact that this is an open game, in the sense discussed further in Section~\ref{Section.Game-values-in-infinite-Hex}, because any Blue win will occur, if at all, by a finite stage of play. In particular, the truncated quadrant games are subject to the theory of ordinal game values. Moreover, this asymmetric game does not allow draws through transfinite play---Blue wins if the Gale tour, constructed as in the proof of Theorem~\ref{Theorem.Finite-hex}, is finite, while Red wins when the Gale tour is infinite.\goodbreak

\begin{corollary}\label{Corollary.Truncated-quadrant-Blue-draw}
If Blue has a first-player winning strategy on every truncated quadrant, then there is no finite advantage in infinite Hex---both players have drawing strategies from any finite position. Indeed, a player has a drawing strategy in any position having an empty opposing-player-relevant quadrant.
\end{corollary}

\begin{proof}
Let us assume that we have a position with an empty quadrant I for some choice of origin, which is a quadrant relevant for a Red win---the other cases are similar. We assume that Blue can win on any truncated quadrant. To force a draw, we shall first direct Blue to play so as to win on that quadrant, inventing imaginary moves for Red when Red does not reply in that quadrant. After building a blue connection of the sides of this truncated quadrant, which will occur after finitely many moves, there will therefore remain a larger truncated quadrant still empty, and we may then direct Blue next to win on that truncated quadrant. And so on. In this way, Blue will build infinitely many disjoint connections of the blue sides, and these will form obstacles to a Red win on the infinite Hex board. So this is an infinite Hex drawing strategy for Blue on any board with an empty quadrant I. A symmetric argument works in the other cases, where there is an empty quadrant relevant for the opponent.
\end{proof}

Notice that if Blue wins on a truncated quadrant, then Blue also wins on all truncated quadrants with a smaller truncation, and so there is no need in Corollary~\ref{Corollary.Truncated-quadrant-Blue-draw} to refer to arbitrarily large as opposed to all truncated quadrants.

Although the assumption that Blue wins on all truncated quadrants may appear at first to be weaker than the trapezoid assumption of Theorem~\ref{Theorem.If-trapezoids-then-draw} and Corollary~\ref{Corollary.Trapezoid-Blue-draw}, since if Blue can win on a trapezoid, then Blue can win on the corresponding truncated quadrant, nevertheless, these two hypotheses are actually equivalent. In this sense, Corollaries~\ref{Corollary.Trapezoid-Blue-draw} and~\ref{Corollary.Truncated-quadrant-Blue-draw} amount to the same fact.

\begin{theorem}\label{Theorem.Truncated-quadrants-iff-trapezoids}
 Blue has a first-player winning strategy on all truncated quadrants if and only if Blue has first-player winning strategies on arbitrarily large trapezoids. 
\end{theorem}

The proof of the forward implication will rely on Theorem~\ref{Theorem.Hex-intrinsically-local}, and so we shall defer the proof of Theorem~\ref{Theorem.Truncated-quadrants-iff-trapezoids} until Section~\ref{Section.Game-values-in-infinite-Hex}. Meanwhile, although the two assertions mentioned in Theorem~\ref{Theorem.Truncated-quadrants-iff-trapezoids} are equivalent, we do not know whether they are true! 

\begin{question}\label{Question.Truncated-quadrants}
Does Blue have a first-player winning strategy on all truncated quadrants?
\end{question}

The question remains open.\goodbreak

\subsection{Alternative strong win condition}

\begin{wrapfigure}[11]{r}{.32\textwidth}\vskip-1ex
\hfill
\includegraphics[width=.28\textwidth]{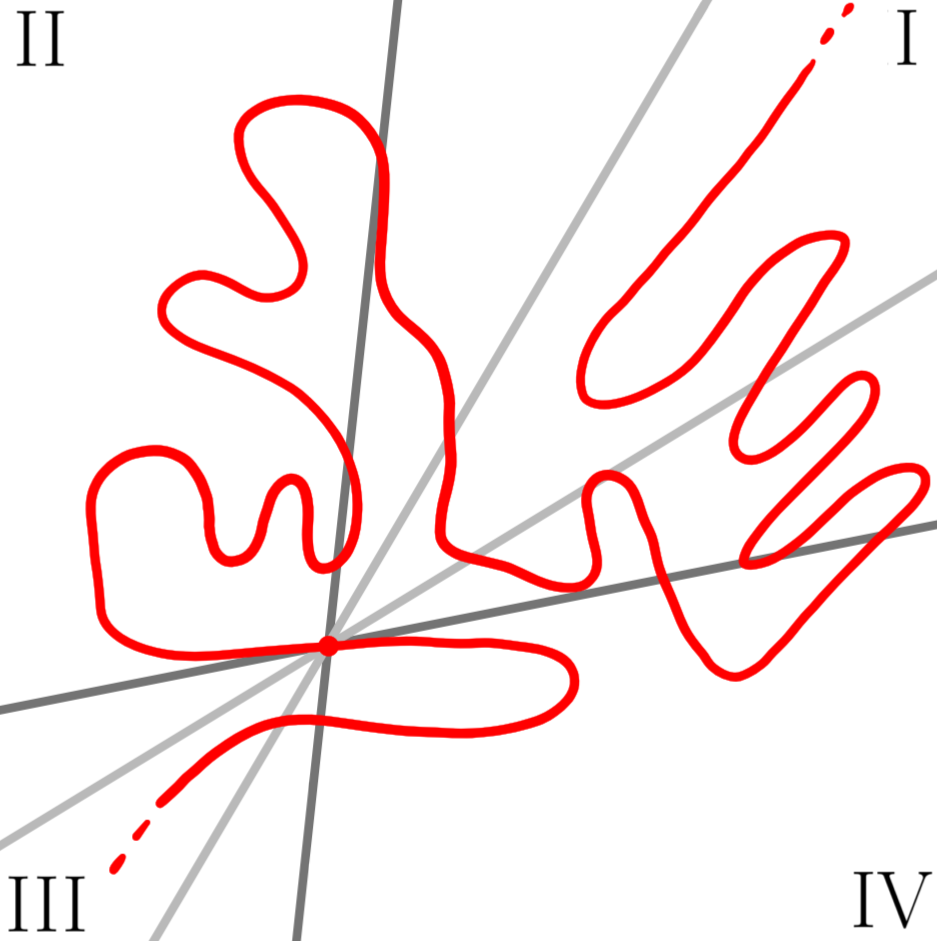}
\captionsetup{style=rightside,font=footnotesize} \caption{A strong Red win}
\label{Figure.sketch-strong-win}
\end{wrapfigure}
Let us now prove a version of the finite advantage conjecture without the trapezoid or truncated quadrant assumptions, but instead with respect to a strengthened winning condition, illustrated in Figure~\ref{Figure.sketch-strong-win}. Specifically, we define that a play of infinite Hex is a \emph{strong win} for Red, if there is a $\mathbb{Z}$-chain of adjacent red hexagons, such that for any positive-slope line in the plane and any point on that line, (i) there is a steeper line through that point, such that the red path touches that steeper line only finitely many times and ultimately crosses from left to right; and (ii) there is a shallower-slope line through that point such that the red path touches it only finitely many times and ultimately crosses it from below to above. Similarly for Blue, with the corresponding change in direction.

If a player wins according to this criterion, then they have also won according to the standard winning condition, which requires them merely to eventually enter the appropriate quadrants determined by any choice of center origin point. So any strong win is also a standard win. But the strong win condition is stronger, since it requires the winning path not merely to enter the quadrant, but to enter a linearly focused ``quadrant'' where the axes form a cone of strictly less than $90^\circ$.

\begin{wrapfigure}{l}{.34\textwidth}\vskip-2ex
\includegraphics[width=.3\textwidth]{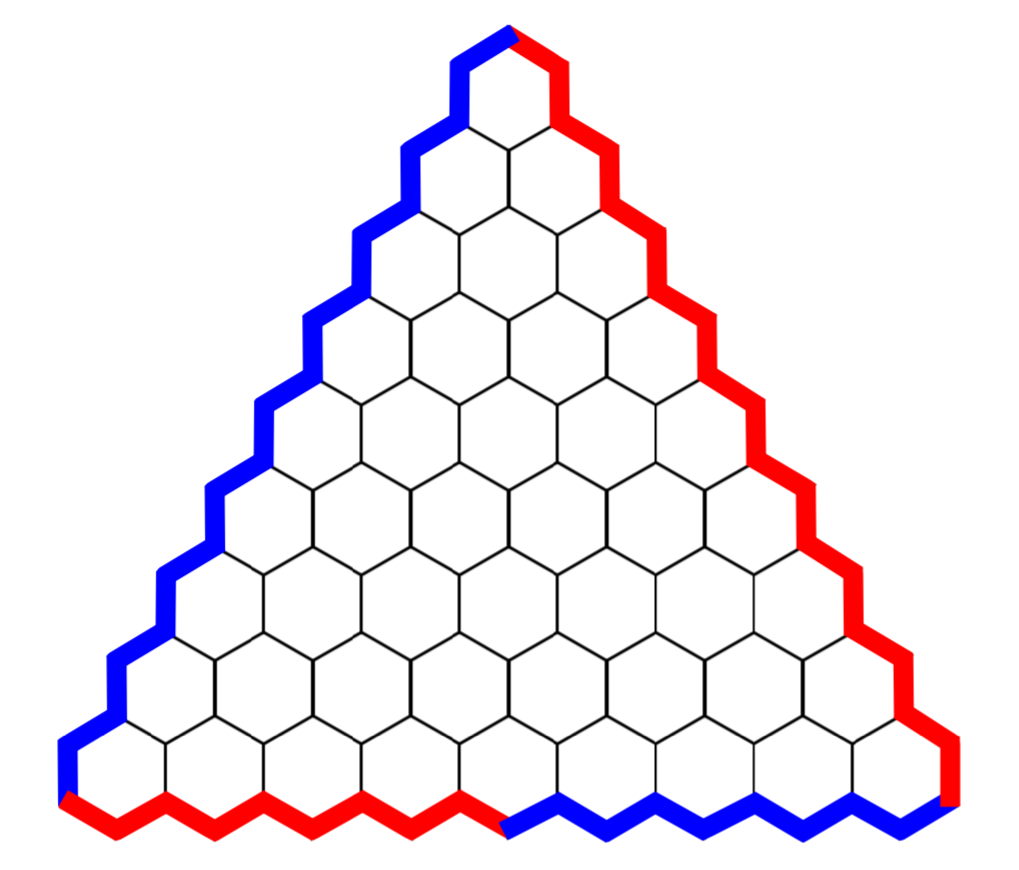}\hfill
\captionsetup{style=leftside,font=footnotesize} \caption{Triangular board}
\label{Figure.triangular-board}
\end{wrapfigure}
Let us begin by recalling that the finite Hex theorem (Theorem~\ref{Theorem.Finite-hex}) applies generally to any symmetric finite Hex board, a board that looks the same with respect to either player. On any finite board, exactly one player will win, and on such a symmetric board, the strategy-stealing argument of Theorem~\ref{Theorem.Finite-hex} shows that it must be the first player who has a winning strategy. In particular, the first player has a winning strategy in any finite equilateral triangular board with boundary coloring as in Figure~\ref{Figure.triangular-board}, as well as on any analogously symmetric isosceles triangular board.

\begin{theorem}
 There is no finite advantage in infinite Hex, when using the strong win condition for winning. That is, for any given finite position in infinite Hex, both players have drawing strategies for play proceeding from that position with respect to the strong win condition for winning.
\end{theorem}

In other words, both players have strategies to prevent the other player from achieving a strong win, even when finitely many stones have already been placed.\goodbreak

\begin{wrapfigure}{r}{.55\textwidth}\vskip-2ex
\hfill
\includegraphics[width=.5\textwidth]{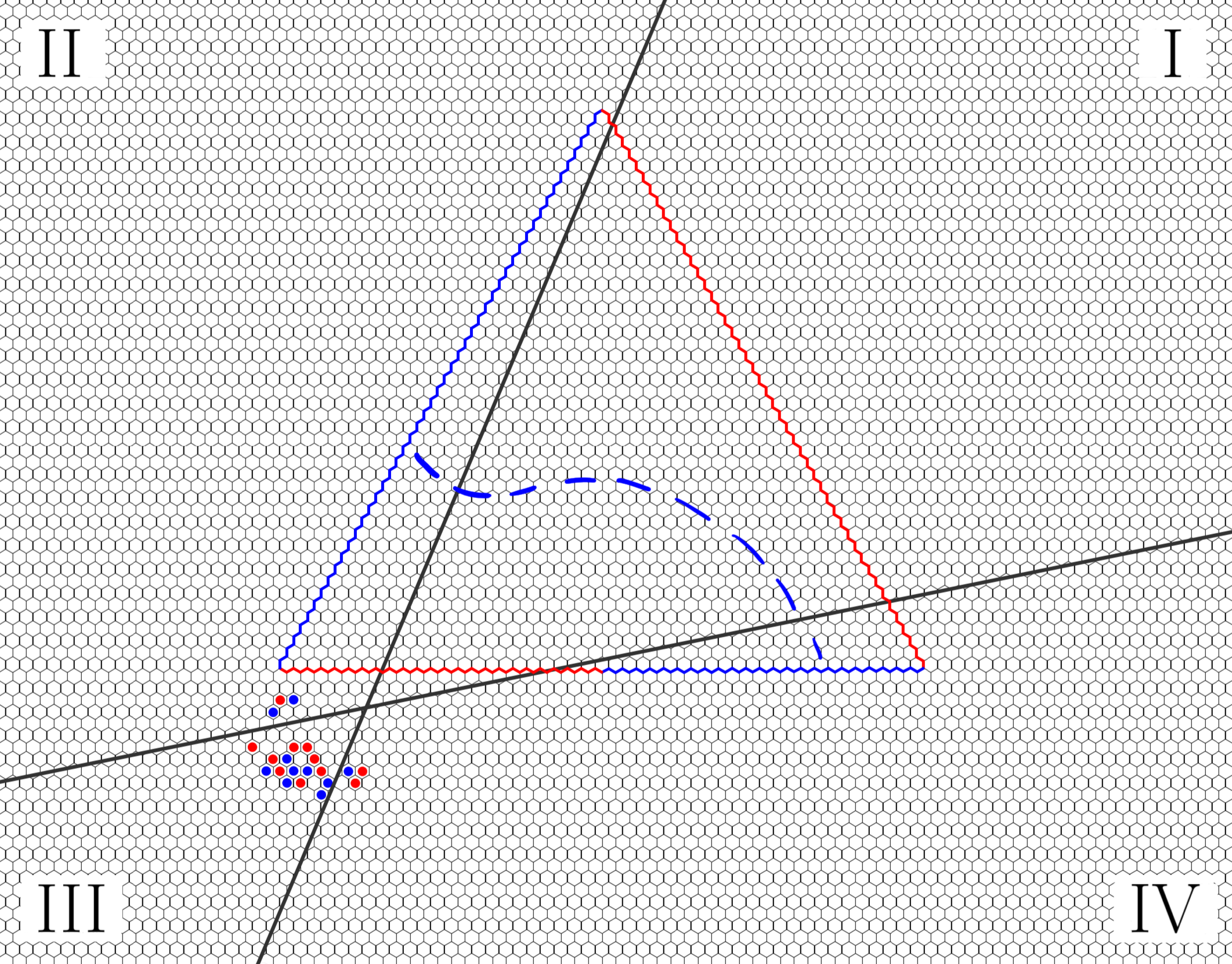}
\captionsetup{style=rightside,font=footnotesize}
\caption{Triangle spanning a linear cone}
\label{Figure.Triangle-proof}
\end{wrapfigure}
\noindent\emph{Proof.} 
It will suffice for us to argue that Blue has such a drawing strategy, for then a symmetric argument applies for Red. Suppose that we have an infinite Hex board position with finitely many stones already placed. For each possible center point and each possible pair of lines with positive rational slope through that point, one nearly vertical and one nearly horizontal, we shall place a certain empty isosceles triangle onto the Hex board, with a horizontal bottom edge having its midpoint lying below the nearly horizontal line and upper vertex above the nearly vertical line, as illustrated in Figure~\ref{Figure.Triangle-proof}. Blue will adopt the strategy of playing on such triangles so as to connect the right-half of the bottom side with the left side of the triangle, thereby forming an obstacle that Red will have to get around in order to achieve a strong win.

Since there are only countably many center points and pairs of lines with positive rational slope, we can place disjoint such triangles for every such choice of center and lines. With the triangles placed, Blue will now play by selecting one of them as ``active,'' making a first move according to the winning strategy for the triangle, and continuing to play in that triangle until the blue sides are connected (inventing imaginary Red moves, if necessary, when Red plays outside). Having won on that triangle, Blue will activate a next (empty) triangle, and proceed to win on that one, and so on. 

Since every choice of center point has infinitely many pairs of lines with rational slope, approaching as close as desired to $90^\circ$, there will be infinitely many triangles paired with any given center point. And so Blue will be able to arrange his choices of which triangles to activate in such a way that after infinitely many moves, he will have activated infinitely many triangles for each center point, with the cone angle of the lines approaching $90^\circ$. After infinite play, therefore, Blue will have placed infinitely many obstacles blocking any particular choice of center and axis lines, and this will prevent a strong win for Red. This is therefore a drawing strategy for Blue with respect to the strong win condition. 
\hfill\fbox{}\medskip\goodbreak 

Notice that this strategy works also with transfinite play. Indeed, Blue will build all the blue obstacles during the first $\omega$ many moves, and so this strategy has already
prevented a strong win for Red, even if Red should make further moves at transfinite stages beyond $\omega$.\goodbreak

\begin{wrapfigure}[14]{r}{.28\textwidth}
\vskip-2ex\hfill
\includegraphics[width=.25\textwidth]{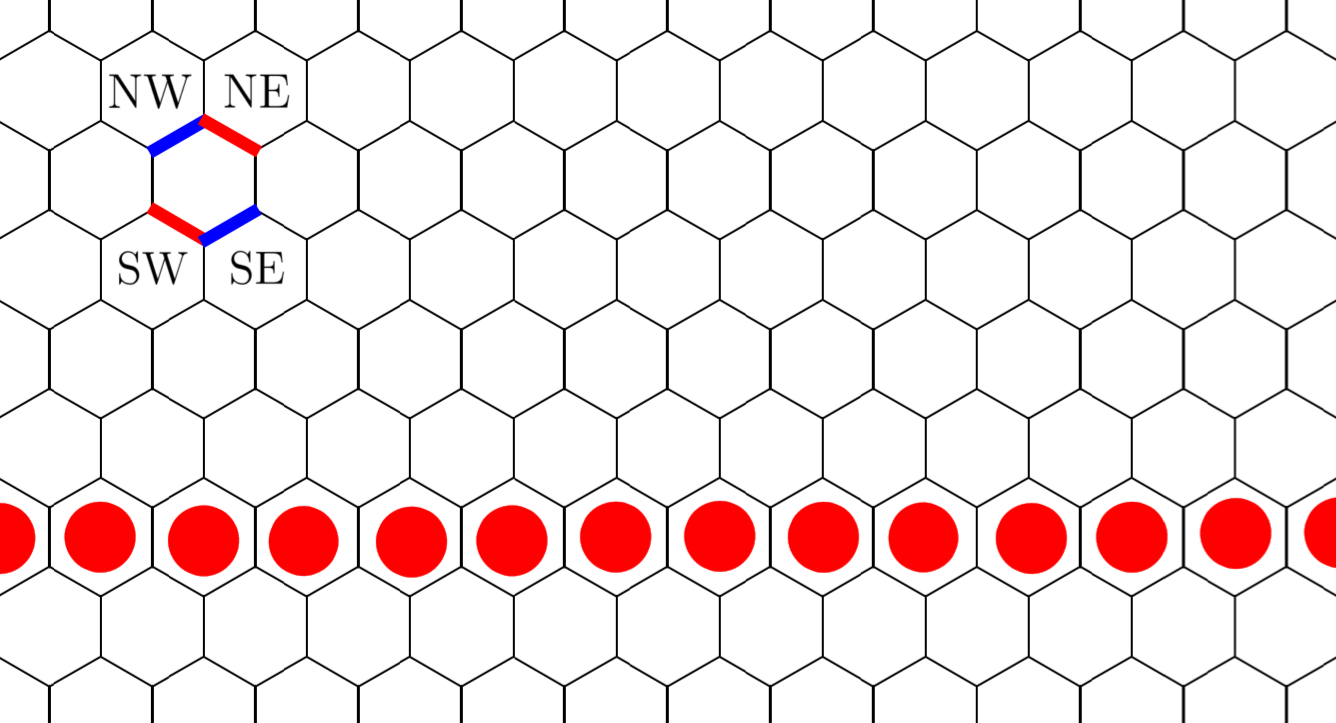}
\captionsetup{style=rightside,font=footnotesize} \caption{Line advantage}
\label{Figure.horizontal-line}
\medskip\medskip\medskip\medskip\medskip
\hfill
\includegraphics[width=.25\textwidth]{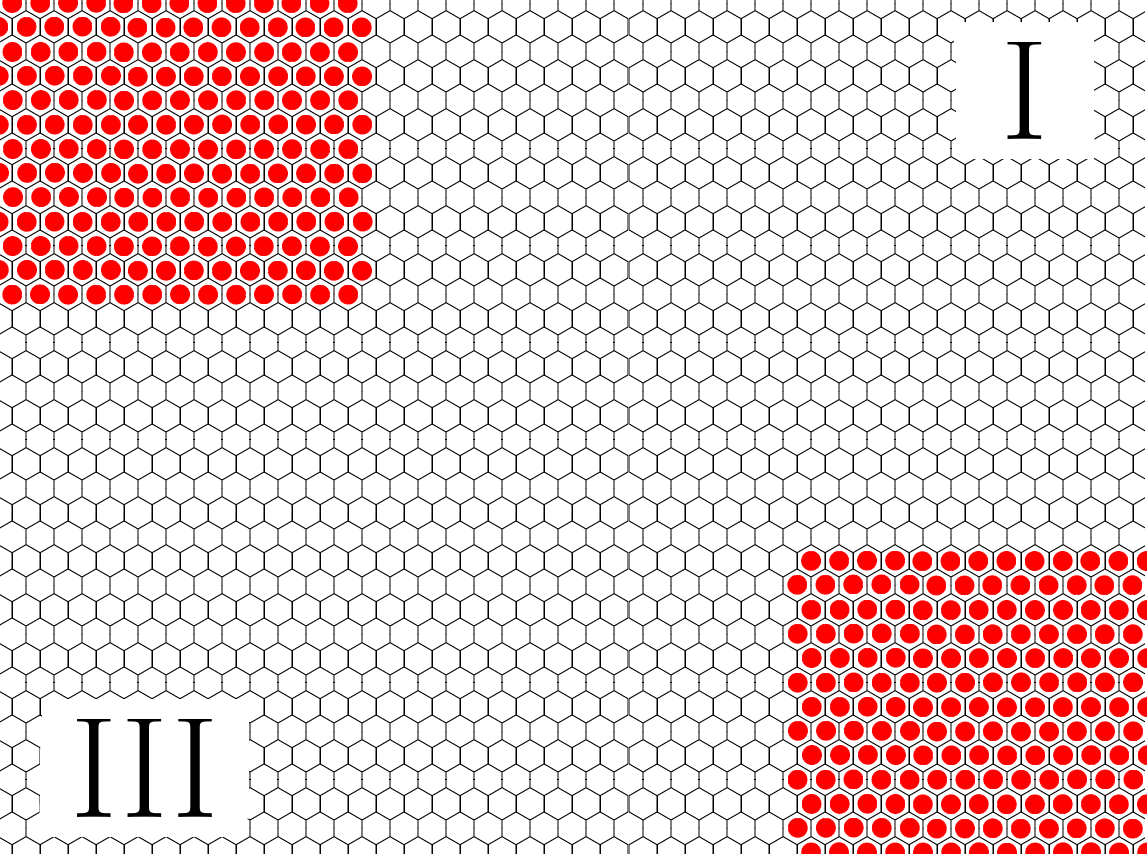}
\captionsetup{style=rightside,font=footnotesize} \caption{Quadrants advantage}
\label{Figure.quadrants-question}
\end{wrapfigure}
The strategy also works for Blue to prevent a some-origin strong win for Red, in the sense that Blue has ensured that there is not a single point at which Red has achieved the strong win condition for lines at that point.

Let us conclude this section with two questions we don't know how to answer. 
\begin{question}\label{Question.Infinite-advantage}
Can Red win infinite Hex when starting with the advantage of an infinite horizontal red line as in Figure~\ref{Figure.horizontal-line}? Or with an infinite red line with negative slope, which does not prevent a Blue win? Or with a horizontal and `vertical' line together?
\end{question}

\enlargethispage{10pt}
\begin{question}\label{Question.Quadrant-advantage}
Can Red win infinite Hex when starting with the advantage of two red quadrants as in  Figure~\ref{Figure.quadrants-question}? \end{question}

\section{Game values in infinite Hex}\label{Section.Game-values-in-infinite-Hex}

The theory of transfinite game values has been explored in infinite chess~\cite{hamkins12}\cite{hamkins14}\cite{hamkins17}, infinite draughts~\cite{Leonessi2021:MSc-dissertation-transfinite-game-values-in-infinite-games}\cite{HamkinsLeonessi:Transfinite-game-values-in-infinite-draughts}, infinite Go~\cite{Hamkins2018.MO299504:Position-in-infinite-Go} and several other infinite games. We refer the reader to that literature for the basic theory of ordinal game values, but we can quickly review the central ideas. An infinite game is \emph{open} for a player, if every winning play for that player is  settled as a definite win for that player by some finite stage of play---that is, after this finite stage, the player cannot avoid winning no matter how (legal) play continues. Infinite chess, for example, is an open game, because the checkmate, when it occurs, does so at a finite stage of play. In any such open game, the ordinal \emph{game value} for that player of a position is defined by recursion: a position has value $0$ for the open player, if the game is already won, in the sense that no subsequent legal moves could prevent a win; if a position is the open player's turn, and a move can be made to a position with value $\alpha$, minimal amongst those with a value, then the value of the original position is $\alpha+1$; if it is the opponent's turn and all plays by the opponent lead to a valued position, then the value of the position is the supremum of those values. The fundamental observation of game values is that if a game position has a game value, then the open player can win by playing the value-reducing strategy---play so as to reduce value; since there is no infinite strictly decreasing sequence of ordinals, eventually the value will become $0$ and the game is won. Similarly, in any open game, if a position has no game value, then the closed player can avoid defeat by playing the value-avoiding strategy.

We find it interesting to note that the game value recursion can be defined for any game, not just open games, and a game position will have a value for a player if and only if the player can force a win in finitely many moves (not necessarily uniformly bounded finite). In particular, if a position in any game has a defined game value, then the value-reducing strategy will be an optimal winning strategy for that player from that position. When a position has no defined value, meanwhile, then the value-avoiding strategy will ensure that the opposing player can avoid all positions that are definite losses at a finite stage of play. In open games, this is the same as avoiding defeat altogether, and so the game value analysis provides a complete strategic account of open games.

Much of the previous literature we mentioned had taken on the goal of exhibiting positions in those games whose game values were various large countable ordinals---as large as could be found. In infinite chess, for example, Evans and Hamkins \cite{hamkins14} presented a position with game value $\omega^3$ and conjectured that every countable ordinal arises as the game value of a position in infinite chess. The situation was slightly improved in \cite{hamkins17} with a position having game value~$\omega^4$. Meanwhile, \cite{hamkins14} proves that every countable ordinal arises as the game value of a position in infinite 3D chess, and our recent work~\cite{Leonessi2021:MSc-dissertation-transfinite-game-values-in-infinite-games, HamkinsLeonessi:Transfinite-game-values-in-infinite-draughts} proves the corresponding result in the case of infinite draughts. In exciting new work, Matthew Bolan \cite{Bolan2023.MO440147:Checkmate-in-omega-moves} has announced that every countable ordinal arises as the game value of a position in infinite chess, thus establishing the conjecture of Evans and Hamkins \cite{hamkins14}.

In this section, we should like to consider the question of which game values arise as the value of positions in infinite Hex. 

\begin{question}
 Which ordinal game values arise for positions in infinite Hex?
\end{question}

\begin{wrapfigure}{r}{.39\textwidth}
\hfill
\includegraphics[width=.36\textwidth]{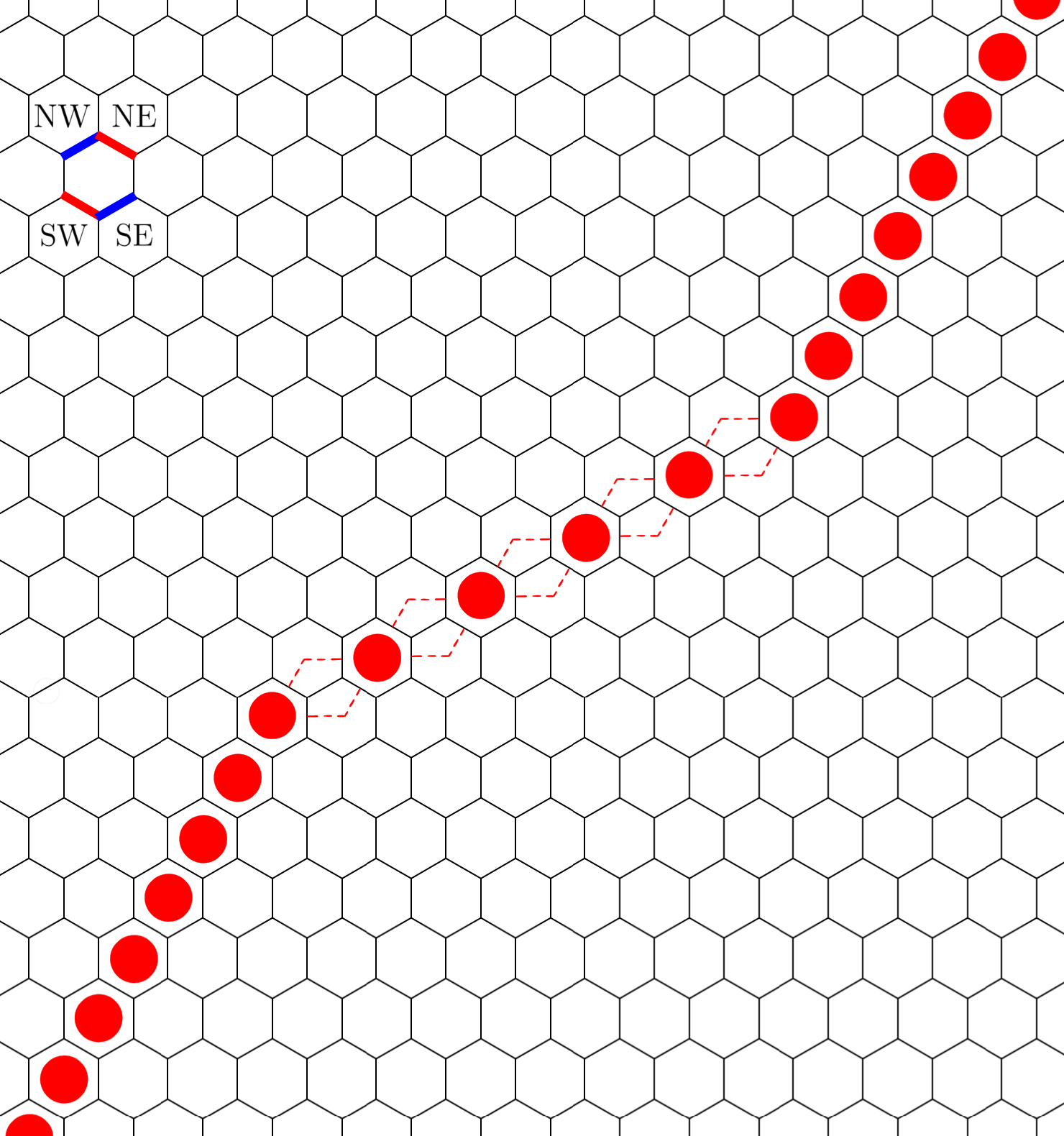}
\captionsetup{style=rightside,font=footnotesize}
\caption{Game value $5$ for Red}
\label{Figure.Game-value-5}
\end{wrapfigure}
The first main difficulty for this question and the theory of game values in infinite Hex is that unlike infinite chess and infinite draughts, infinite Hex is not in general an open game. The winning condition of infinite Hex, starting from an empty board, is inherently infinitary---one does not win after only finitely many moves. Thus, game values do not tell the whole story, although in positions for which they are defined the value-reducing strategy will still be a winning strategy.

\enlargethispage{20pt}
For a position in infinite Hex to have a defined game value for Red, say, then by the value-reducing strategy, Red would be able to force a win in finitely many moves. So the position would have to already contain infinitely many stones, and indeed it would have to contain a mostly completed winning $\mathbb{Z}$-chain, missing only finitely many intermediate connecting stones. Red aims to connect the two ends of the chain together and thereby win. The position in Figure~\ref{Figure.Game-value-5}, for example, has game value $5$, since Red can aim to complete the five bridges, which will take $5$ moves, and there is nothing Blue can do to prevent it. One can similarly create positions in infinite Hex with any desired finite game value.\goodbreak

\subsection{Infinite Hex is intrinsically local}
What we aim to prove about infinite Hex is that only finite game values occur. In this sense, the truly transfinite game value phenomenon, which has been so abundantly exhibited for infinite chess, infinite draughts, and infinite Go, simply does not occur in infinite Hex.\goodbreak

\begin{theorem}\label{Theorem.Hex-finite-values}
Only finite game values occur for positions in infinite Hex. 
\end{theorem}

We shall indeed establish a much stronger result about such positions. Namely, we claim that every position of infinite Hex for which a player can force a win in finitely many moves (not necessarily uniformly bounded)---these are precisely the positions with an ordinal game value, including transfinite game values---is \emph{intrinsically local}, meaning that there is a finite part of the board such that the winning player can win by playing only in that region and paying attention only to moves of the opponent in that region. It follows that the game value of that position is in fact finite, bounded by (half) the size of that finite region.

\begin{theorem}\label{Theorem.Hex-intrinsically-local}
If a position in infinite Hex has a game value, then it is intrinsically local.
\end{theorem}

Theorem~\ref{Theorem.Hex-finite-values} follows as an immediate consequence of Theorem~\ref{Theorem.Hex-intrinsically-local}, because if a player can win by playing on and responding to plays on a fixed finite subboard, then the game value is bounded by half the size of that subboard, since the game will be over and won (with correct play) when the subboard is filled.

\medskip\noindent\emph{Proof of Theorem~\ref{Theorem.Hex-intrinsically-local}.} 
Suppose that we are given an infinite Hex position with a defined game value $\alpha$ for Red. We shall prove by induction on $\alpha$ that this position is intrinsically local. If the game value $\alpha$ is zero, then the game is already won for Red, of course, and we can take the finite subboard to be empty. So we may assume the game value $\alpha$ is not zero.

\begin{wrapfigure}{r}{.28\textwidth}
\vskip-2ex\hfill
\includegraphics[width=.25\textwidth]{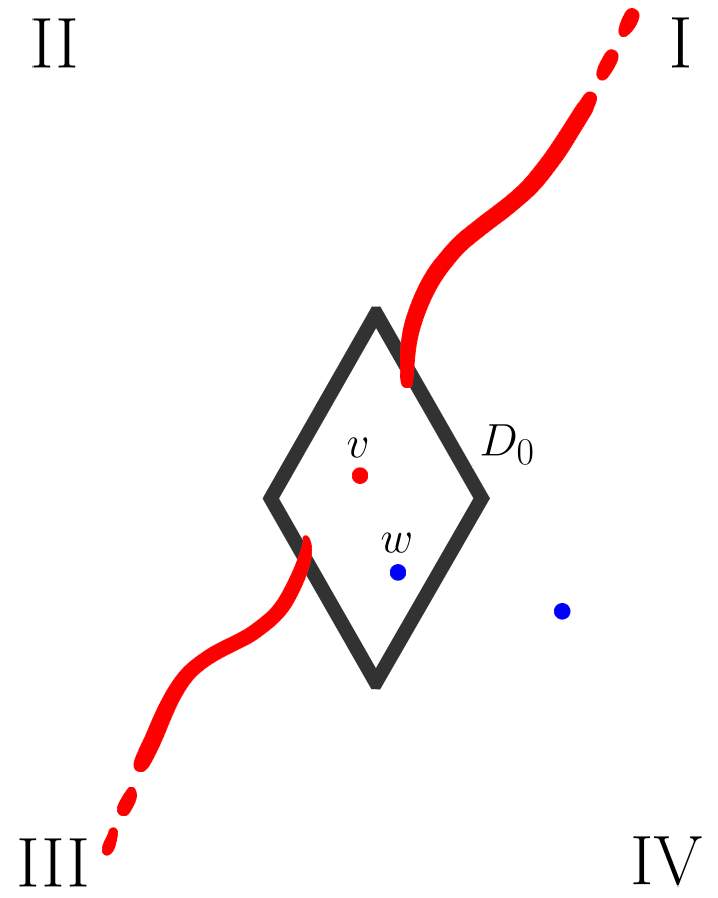}
\captionsetup{style=rightside,font=footnotesize}
\caption{Blue plays outside $D_0$}
\label{Figure.Hex-intrinsically-local-D_0}
\end{wrapfigure}
If it is Red's turn to play, then Red can play to a position with strictly lower game value, and by induction that position is intrinsically local. By augmenting that finite subboard with the Hex tile of the initial move, we have thus found a finite subboard fulfilling the intrinsically local condition for the original position.

The remaining case occurs when the game value is nonzero and it is Blue's turn to play. Every play by Blue leads to some position with game value at most $\alpha$. Let Red imagine temporarily that perhaps Blue might place a stone on some (arbitrary) empty Hex tile $w$. Using the value-reducing strategy, Red would be able to respond to this move with a play on some Hex tile $v$, resulting in a position with strictly lower game value than $\alpha$, and hence intrinsically local. So there is some finite region $D_0$ supporting the intrinsically local winning play after Blue $w$ and Red $v$, as in Figure~\ref{Figure.Hex-intrinsically-local-D_0}. We may assume $w,v\in D_0$. If in the actual game, Blue should by some miracle happen to play on tile $w$, as Red had imagined, then indeed Red could win by playing only on $D_0$, and so this finite subboard would have exhibited the desired intrinsic locality condition.\goodbreak

But of course, in the actual game, perhaps Blue doesn't play on tile $w$---there might be infinitely many other tiles on which Blue could play. If Blue were to play outside $D_0$, however, then Red can actually ignore that move, and proceed as though Blue had played on $w$, placing a red stone on tile $v$ and then proceeding to play on and respond to moves only in $D_0$, inventing imaginary moves in $D_0$ whenever Blue plays outside, as illustrated in Figure~\ref{Figure.Hex-intrinsically-local-D_0}.\goodbreak

A more troublesome case occurs, however, if in the actual game Blue should begin by placing a stone in $D_0$, but not on tile $w$. For each tile $w_i$ in $D_0$ in the original position (taking $w=w_0$), if Blue were to place a blue stone at $w_i$, then by the value-reducing strategy, there is a Red reply on some tile $v_i$, leading to a position with game value strictly less than $\alpha$, which is therefore intrinsically local. So there would be some finite subboard $D_i$ supporting intrinsically local winning
\begin{wrapfigure}{l}{.34\textwidth}
\includegraphics[width=.3\textwidth]{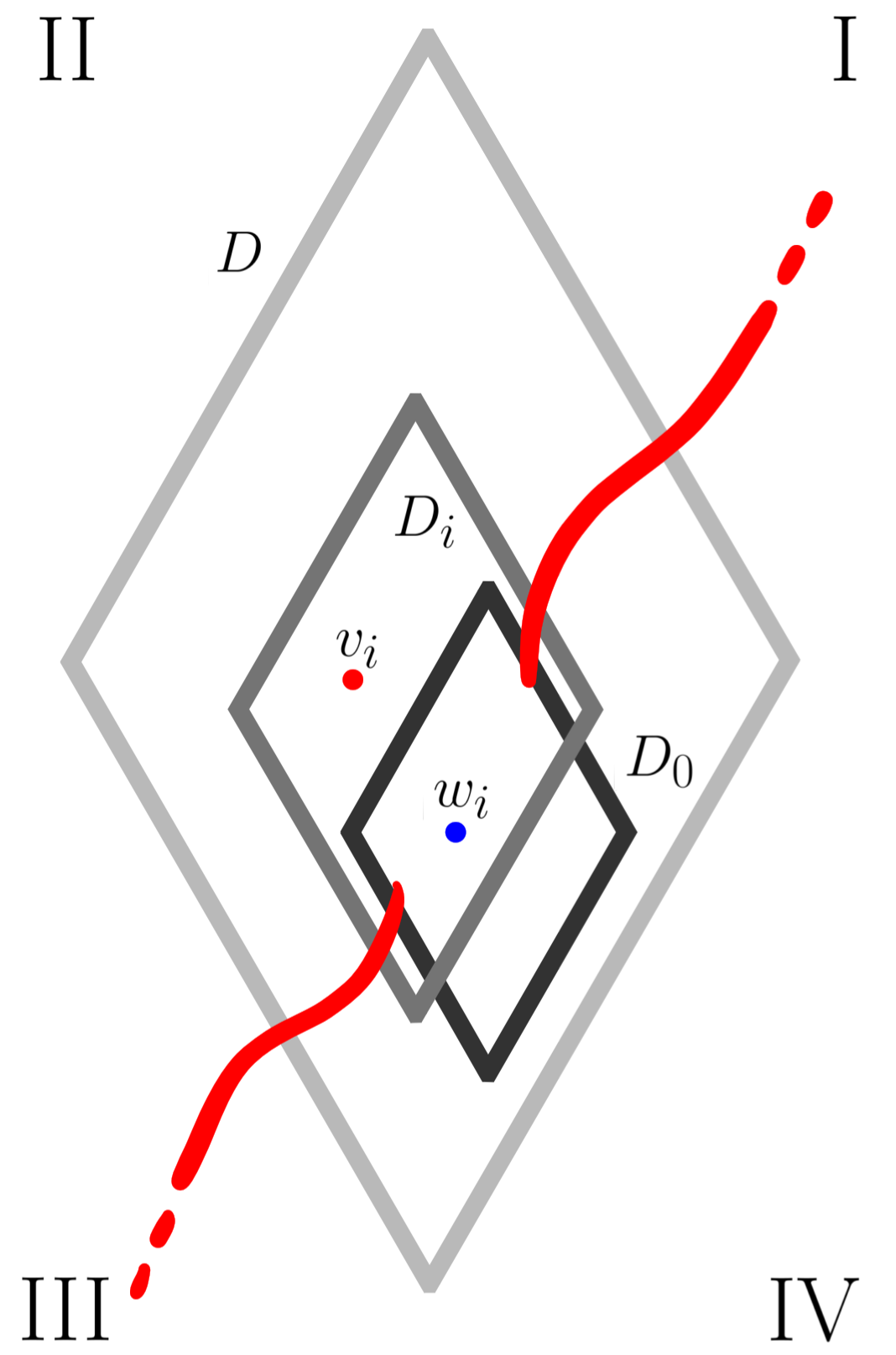}
\hfill
\captionsetup{style=leftside,font=footnotesize}
\caption{Blue plays in $w_i$}
\label{Figure.Hex-intrinsically-local-D}
\end{wrapfigure}
play by Red after Blue $w_i$ and Red $v_i$. We may assume $w_i,v_i\in D_i$. 

Let $D=\bigcup_{w_i\in D_0} D_i$, which is a finite union of finite sets, hence finite. We claim that this set shows that the original position is intrinsically local. If in the actual game, Blue plays outside $D_0$, then Red can win as we described by pretending that he had played on tile $w_0$ and responding only within $D_0$, which is part of $D$. If alternatively, Blue plays on some tile $w_i$ in $D_0$, then Red can respond on tile $v_i$ and play afterward only in $D_i$, which is also part of $D$, as illustrated in Figure~\ref{Figure.Hex-intrinsically-local-D}. So $D$ witnesses that the position was an intrinsically local win for Red. 
\hfill\fbox{}\medskip\goodbreak 

As an application of Theorem~\ref{Theorem.Hex-intrinsically-local}, let us now provide the promised proof of Theorem~\ref{Theorem.Truncated-quadrants-iff-trapezoids}, asserting that Blue can win all truncated quadrants if and only if he can win arbitrarily large finite trapezoids.

\smallskip\noindent\emph{Proof of Theorem~\ref{Theorem.Truncated-quadrants-iff-trapezoids}.} 
We explained in the proof of Theorem~\ref{Theorem.If-trapezoids-then-draw} that by ``arbitrarily large'' trapezoids, we mean that the short Red side can be made as long as possible, with the other sides being chosen large enough so that Blue has a first-player winning strategy. And it is clear as we mentioned earlier that if Blue can win on arbitrarily large trapezoids in this sense, then Blue can win on all truncated quadrants, simply by placing a trapezoid within the truncated quadrant and playing only within it.

\begin{wrapfigure}{r}{.4\textwidth}
\vskip-1ex\hfill
\includegraphics[width=.35\textwidth]{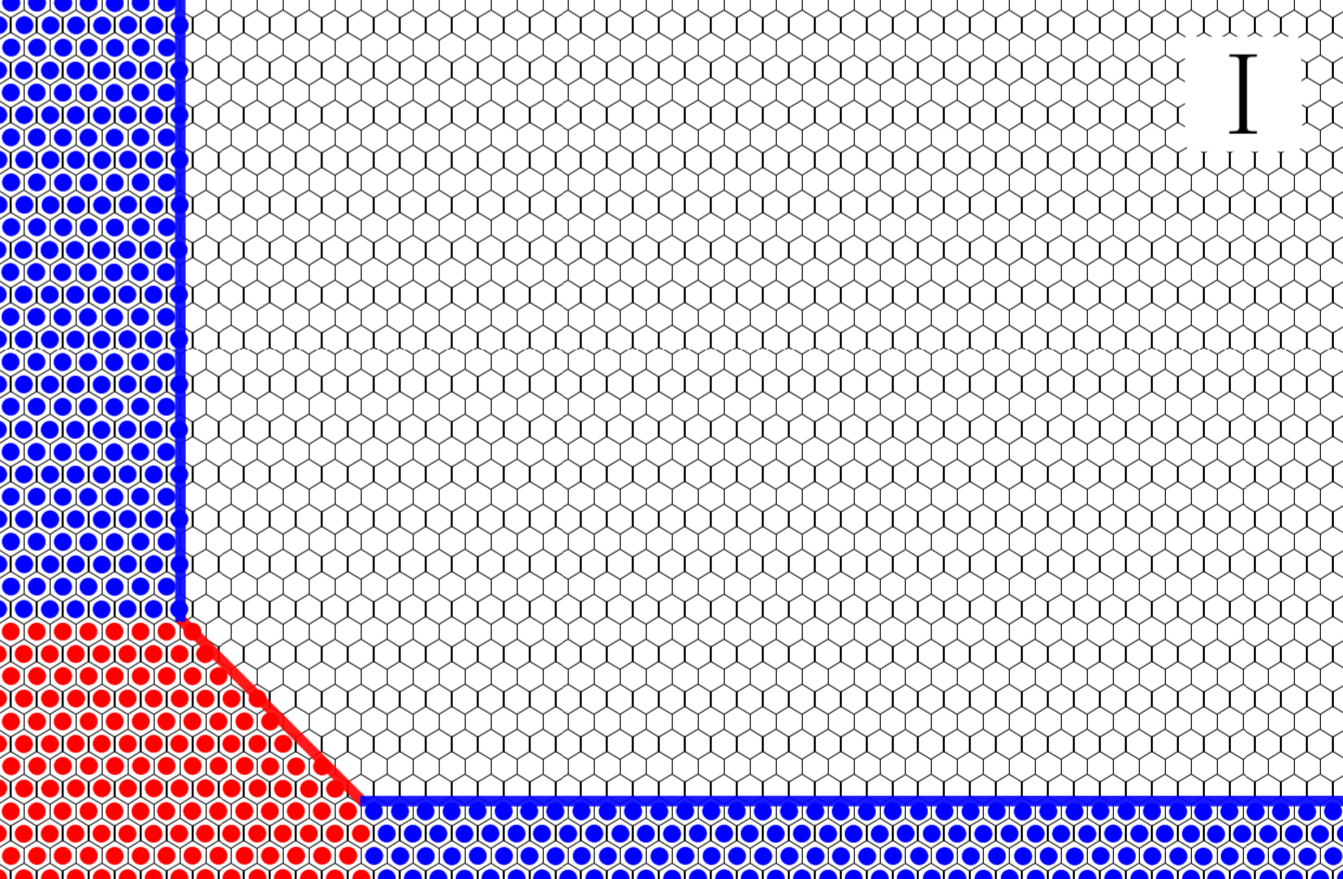}
\captionsetup{style=rightside,font=footnotesize}
\caption{Truncated quadrant as infinite Hex position}
\label{Figure.Infinite-trapezoid-filled}
\end{wrapfigure}
It remains to prove the other direction, that if Blue can win on all truncated quadrants, then Blue can win on arbitrarily large trapezoids. Suppose that Blue has a first-player winning strategy on a given truncated quadrant. Let us interpret this truncated quadrant game as an equivalent instance of infinite Hex, simply by filling in the rest of the infinite Hex board with stones as indicated in Figure~\ref{Figure.Infinite-trapezoid-filled}. The point of this translation is that any win by Blue in the truncated quadrant game will be a win for Blue in the corresponding infinite Hex game and vice versa---so the games are equivalent. Since we have assumed that Blue can win the truncated quadrant game, which is an open game, this is therefore a position in infinite Hex that Blue can win in finitely many moves. Therefore, by Theorem~\ref{Theorem.Hex-intrinsically-local}, the game is intrinsically local, and so there is a finite subboard of this position such that Blue can win by playing on and responding to moves on only that finite subboard. This finite subboard will be included in a trapezoid contained within the truncated quadrant, with a short red side length at least as long as the truncated segment of the truncated quadrant. In short, if Blue wins on a truncated quadrant, then he wins on a large enough finite portion of it, which we can take without loss of generality to be a trapezoid. So if Blue can win on all truncated quadrants, then Blue can win on arbitrarily large trapezoids, as desired.
\hfill\fbox{}\medskip\goodbreak 

Notice that the non-constructiveness of the proof of Theorem~\ref{Theorem.Hex-intrinsically-local} is mirrored by the fact that game values of infinite Hex positions are not in general computable, nor computably enumerable. Clearly there is no computable procedure to compute in finite time whether a given position (presented as an oracle) is intrinsically local, since any such computation would have inspected at most finitely much of the board, and no finite part of a board position is sufficient to determine whether it is intrinsically local.

\subsection{A generalization to simple stone-placing games}

We should like now to generalize the intrinsic locality phenomenon we have established for infinite Hex to a somewhat more general game-theoretic context, proving it for what we call the \emph{simple stone-placing} games, a class of games including infinite Hex and other games, with the key features that stone placements are irreversible---once the stone is placed it never moves again---and extra stone placements are never disadvantageous.

Specifically, a simple stone-placing game is one played on a (possibly infinite) game board upon which the players, taking turns, place their colored stones, with each player striving to create a winning configuration with their stones from a specified set of sufficient winning configurations for that player. Thus, a simple stone-placing game is specified by a set $X$ of possible board locations and the respective sets of sufficient winning configurations $\mathcal{R},\mathcal{B}\subseteq \mathcal{P}(X)$. Player Red is striving to place red stones at every location of some winning red set $R\in\mathcal{R}$ and player Blue is striving to occupy a winning blue set $B\in\mathcal{B}$. It is fine for a player to have extra stones of their color outside the winning configuration. In order that not both players win at a given position, we insist that there are no disjoint winning sets $R\in\mathcal{R}$ and $B\in\mathcal{B}$. The game might allow draws, if some partitions of the board $X=X_R\sqcup X_B$ into red and blue sets have the property that no subset of $X_R$ is in $\mathcal{R}$ and no subset of $X_B$ is in $\mathcal{B}$.\footnote{Finite simple stone-placing games can be seen as isotone set coloring games, as defined in \cite{vanRijswijck2006}, with the requirement of nondisjointness of winning conditions.} 
Notice that a simple stone-placing game is open for a player if the winning configurations for that player can be specified by a set of finite winning configurations, for then the player will win, if at all, after finitely many moves.

In standard play, a simple stone-placing game proceeds for $\omega$ many moves or until the board is full, although the framework allows for transfinite play, provided that the turn order is specified for limit ordinals.

\enlargethispage{20pt}
Infinite Hex is a simple stone-placing game, since the winning Red configurations are specified by the standard infinite Hex winning condition and similarly for Blue---Theorem~\ref{Theorem.One-player-wins}, which states that at most one player wins for any coloring of the infinite board, proves exactly that infinite Hex satisfies the nondisjointness of winning conditions required by the definition of simple stone-placing games. We mention two further examples of stone-placing games that generalize Hex. In the Shannon switching game, the players choose edges of a given graph with two distinguished nodes so that player Short aims to join such nodes, while player Cut strives to prevent that---Hex can be seen as a Shannon game on a particular graph. The game of Y is played on a triangular board as the one seen before in Figure~\ref{Figure.triangular-board}, but without the boundary coloring, so that each player aims to construct a Y-like chain connecting all three sides of the board---also in this case it is clear that the winning sets cannot be pairwise disjoint, and it is not hard to find a position of Y which is game-theoretically equivalent to the initial position of Hex.\goodbreak

Meanwhile, let us mention that several other games that are or can be played by placing colored stones on a board, such as Go, Gomoku, tic-tac-toe, Othello, are not simple stone-placing games according to our definition. In both Gomoku (five-in-a-row) and tic-tac-toe---which are known as \emph{positional} games, as both players aim to occupy exactly the same winning configurations---it is the \emph{first} player to create the winning pattern that wins the game, whether or not the other player could also have created such a winning configuration with further play, and therefore these games do not fulfill the nondisjointness requirement in our definition. Similarly, the games of Go and Othello are not simple stone-placing games, because in the game of Go, stones are sometimes removed from the board during play (when they are captured) and in Othello, stones can sometimes change color. The connection game of Twixt admits a natural infinitary variant, and although this seems not to be a simple stone-placing game, nevertheless it appears that a version of Theorem~\ref{Theorem.Hex-intrinsically-local} will go through to exhibit intrinsic locality for infinite Twixt.  

The main point we should like to make about simple stone-placing games is that both Theorems~\ref{Theorem.Hex-finite-values} and~\ref{Theorem.Hex-intrinsically-local} generalize to this context. 

\begin{theorem}\label{Theorem.Stone-placing-intrinsically-local}
If a position in a simple stone-placing game has a game value, then it is intrinsically local, and consequently, the game value must be finite.
\end{theorem}

\begin{proof}
In the proof of Theorem~\ref{Theorem.Hex-intrinsically-local} we refer to the winning condition of Red only in terms of the finite subboard realizing Red's intrinsically local play. In that proof, the ability of Red to win by playing only within the subboard $D$, even when Blue initially plays outside of $D_0$, relies on the impossibility for Blue to make threats disjoint from $D$, because if Red succeeds there Blue has lost---this property is present also in simple stone-placing games thanks to the nondisjointness of winning conditions. Thus, we can replicate the previous argument mutatis mutandis to prove this theorem. The second part of this theorem follows just as Theorem~\ref{Theorem.Hex-finite-values} was a consequence of Theorem~\ref{Theorem.Hex-intrinsically-local}.
\end{proof}

\section{A multiple-stone playing variation of infinite Hex}\label{Section.Multiple-stone-play}

Timothy Gowers~\cite{Gowers2021.Twitter:Multiple-stone-Hex}  suggests a variation of infinite Hex in which one player---let's say Red---is allowed to place two stones on each turn, whilst Blue places only one. This version of the game has a hint of the Angel and Devil game. 

\begin{question}[Gowers]\label{Question.Multiple-stones}
If Red can place two stones on each turn, whilst Blue places only one, is infinite Hex still a draw? 
\end{question}

The answer to this question is strongly negative, as shown later in Theorem~\ref{Theorem.2-for-1-win}, and we proceed towards that result in stages.\goodbreak

Joshua Zelinsky observed that seven-for-one play is a clear win for Red, since every Blue play can be immediately surrounded and isolated by Red, using at most six stones, with one stone left over for making progress on a winning path. Harrison Brown~\cite{Brown2022.Twitter:Multiple-stone-Hex} improved this to the case of three-for-one as follows:

\begin{theorem}[Brown]\label{Theorem.3-for-1}
In the infinite Hex variation in which Red places three stones on each turn, while Blue places only one, Red has a winning strategy.
\end{theorem}\goodbreak

\begin{wrapfigure}[9]{r}{.34\textwidth}
\vskip-2ex\hfill
\includegraphics[width=.3\textwidth]{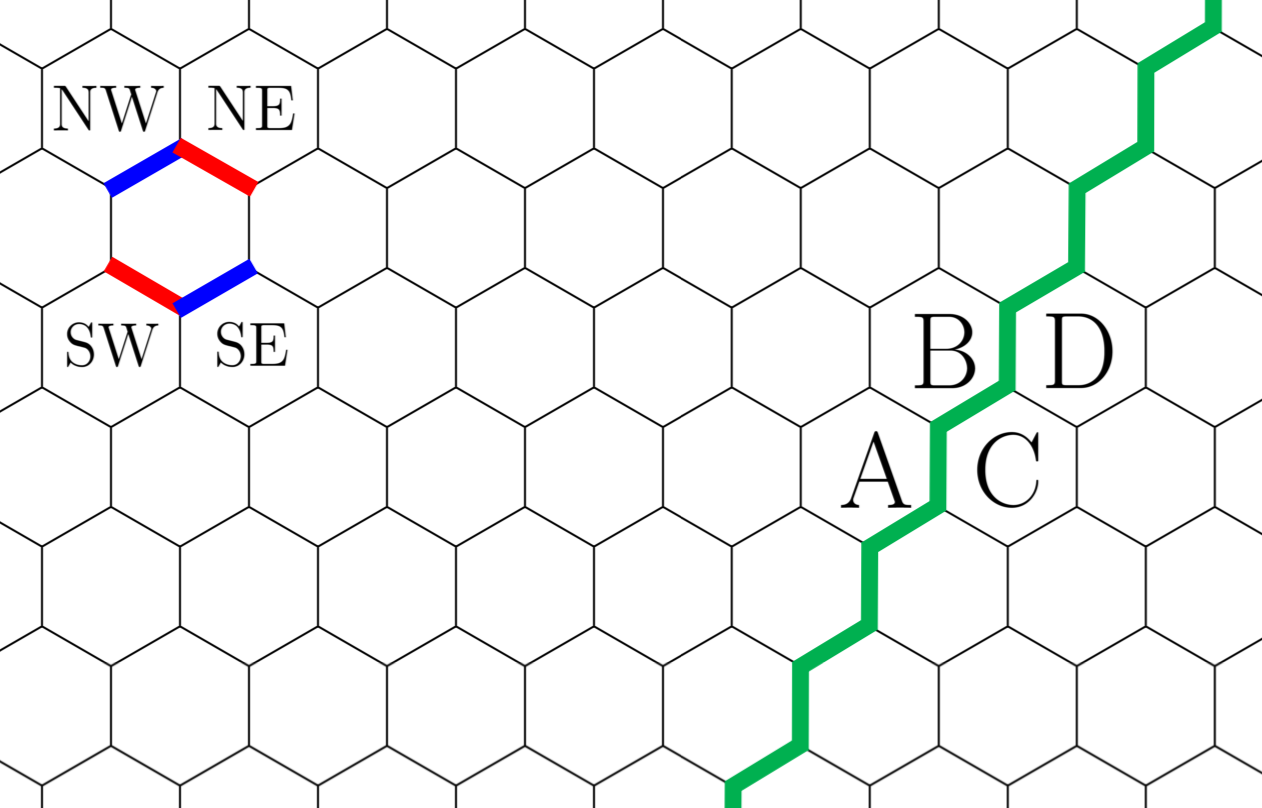}
\captionsetup{style=rightside,font=footnotesize}
\caption{Red's winning 3-for-1 strategy}
\label{Figure.channel-strategy}
\end{wrapfigure}
\noindent\emph{Proof.} 
At the outset of the game, let Red fix a particular diagonal stretching from infinity at lower left to infinity at upper right, as indicated in green in Figure~\ref{Figure.channel-strategy}, as well as a particular tile to be considered as origin. Red will aim to win by playing on tiles adjacent to this diagonal only. We describe Red's strategy by assigning to each tile adjacent to the diagonal the two tiles adjacent to it on the opposite side of the diagonal---with the notation of Figure~\ref{Figure.channel-strategy}, to tile $B$ we assign the tiles $C$ and $D$, while to tile $C$ we assign the tiles $A$ and $B$. 

\begin{wrapfigure}[11]{l}{.38\textwidth}
\vskip-1ex
\includegraphics[width=.34\textwidth]{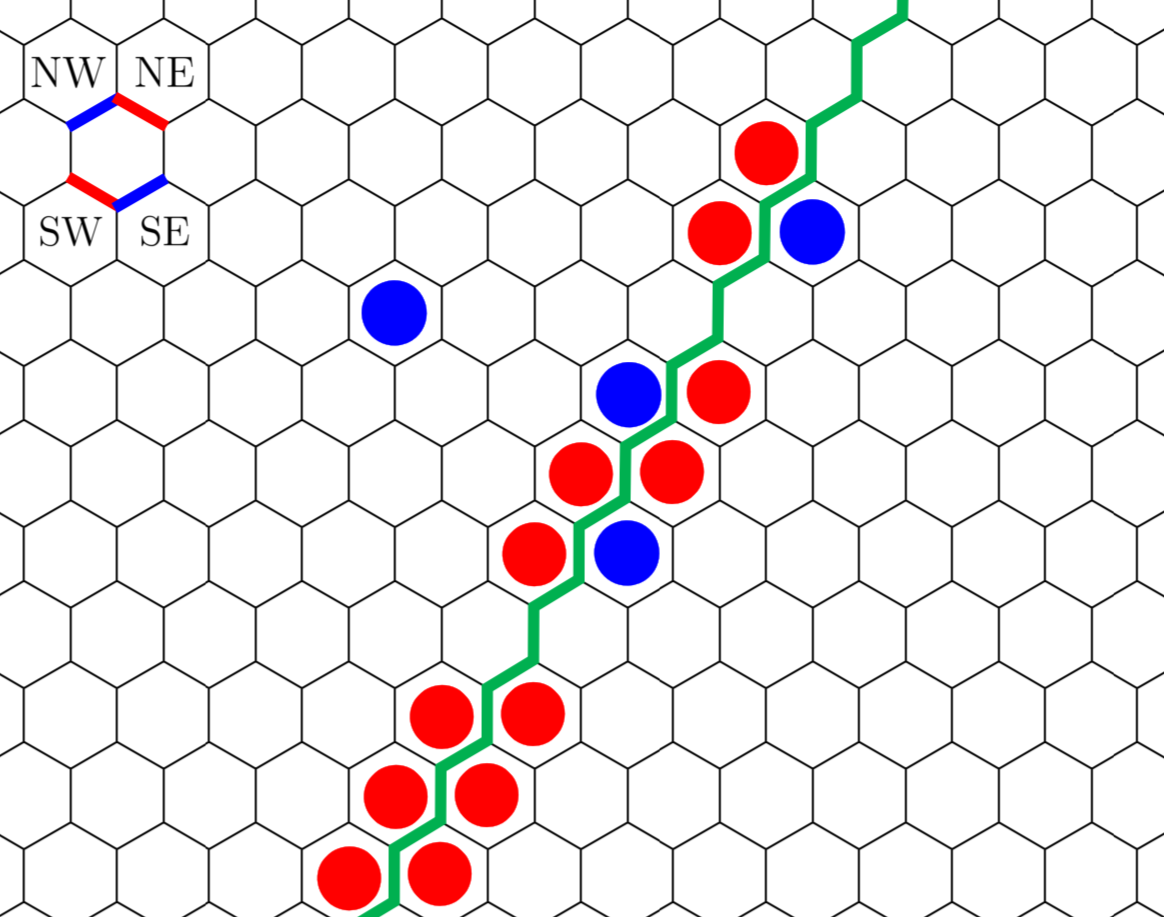}\hfill
\captionsetup{style=rightside,font=footnotesize}
\caption{Partial 3-for-1 play}
\label{Figure.3-for-1}
\end{wrapfigure}
\noindent If Blue should ever place a stone on a tile adjacent to the green diagonal, then Red can reply with two stones on the tiles assigned to such tile, building a protected detour around the blue stone. With the third stone, Red will play at the unfilled tile adjacent to the diagonal which is closest to the origin, as illustrated in Figure~\ref{Figure.3-for-1}. If Red finds that one or both of the tiles assigned to a move are already filled, then he can simply play the extra stones elsewhere adjacent to the diagonal.
In this way, Blue will not be able to cross the diagonal with two adjacent stones, and so the blue stones can never form a true obstacle across the diagonal. So Red will eventually complete a winning chain along the diagonal.
\hfill\fbox{}\medskip\goodbreak 

\begin{wrapfigure}{r}{.36\textwidth}
\vskip-1ex\hfill
\includegraphics[width=.32\textwidth]{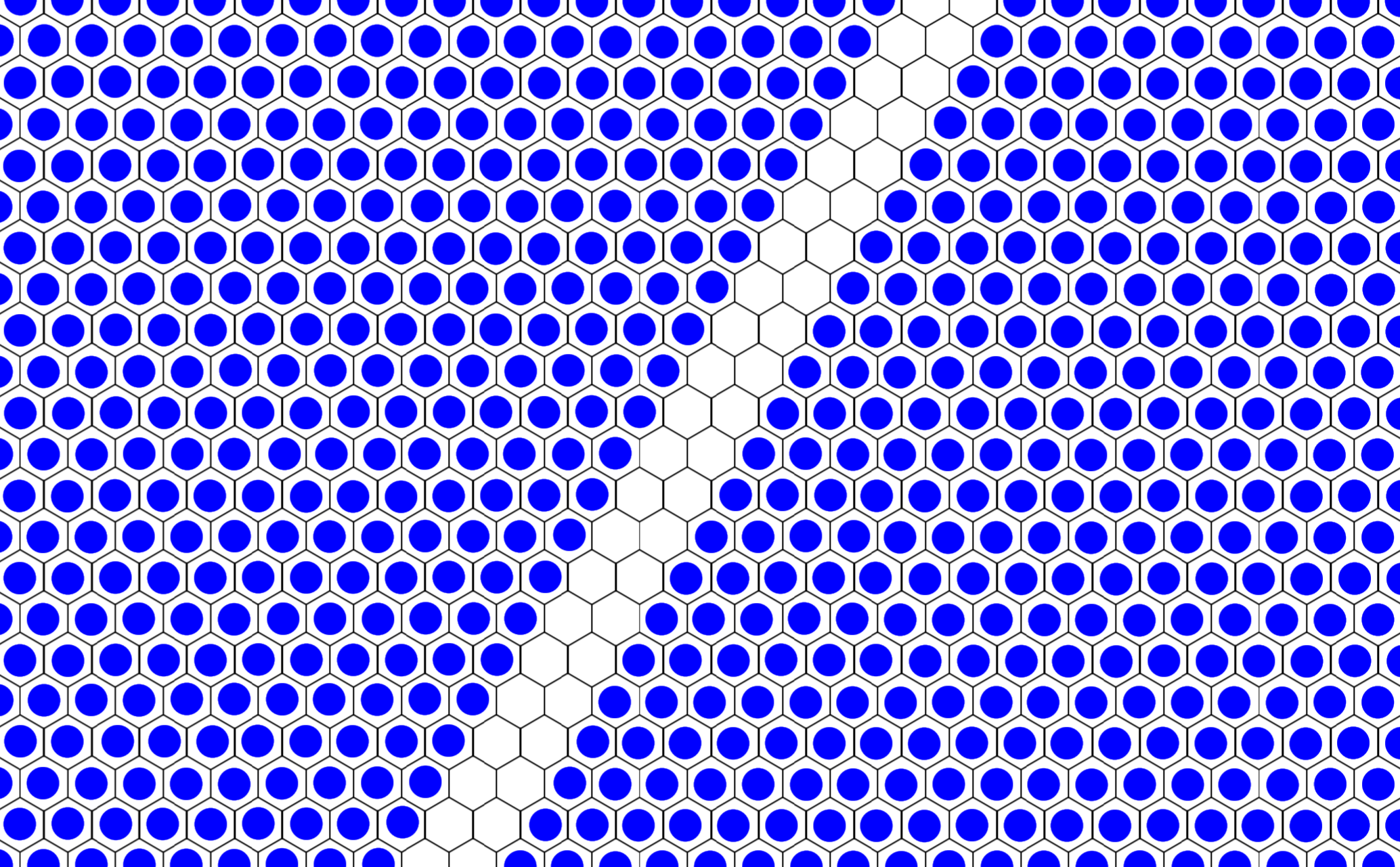}
\captionsetup{style=leftside,font=footnotesize}
\caption{Red wins 3-for-1 play}
\label{Figure.3-for-1-channel-biased}
\end{wrapfigure}
The argument succeeds in the case that normal play is two-for-one, as long as infinitely often Red is allowed to play three stones. This would work even if Blue were the one to decide when Red was allowed to place three stones, provided that he in fact allowed this infinitely often. The argument even accommodates infinitely many blue stones in the initial configuration, as in Figure~\ref{Figure.3-for-1-channel-biased}.

Notice that Red can adapt the strategy described above to win the two-for-one variant of the game on finite boards such as the trapezoids considered in Theorem~\ref{Theorem.If-trapezoids-then-draw}---when Red only needs to build a finite winning chain, there is no need for a third stone at each turn.

We will now refine the argument above to prove the two-for-one case, building on an idea kindly proposed by the anonymous referee.\goodbreak

\newpage
\begin{wrapfigure}[14]{r}{.65\textwidth}\vskip-0ex
\hfill
\includegraphics[width=.6\textwidth]{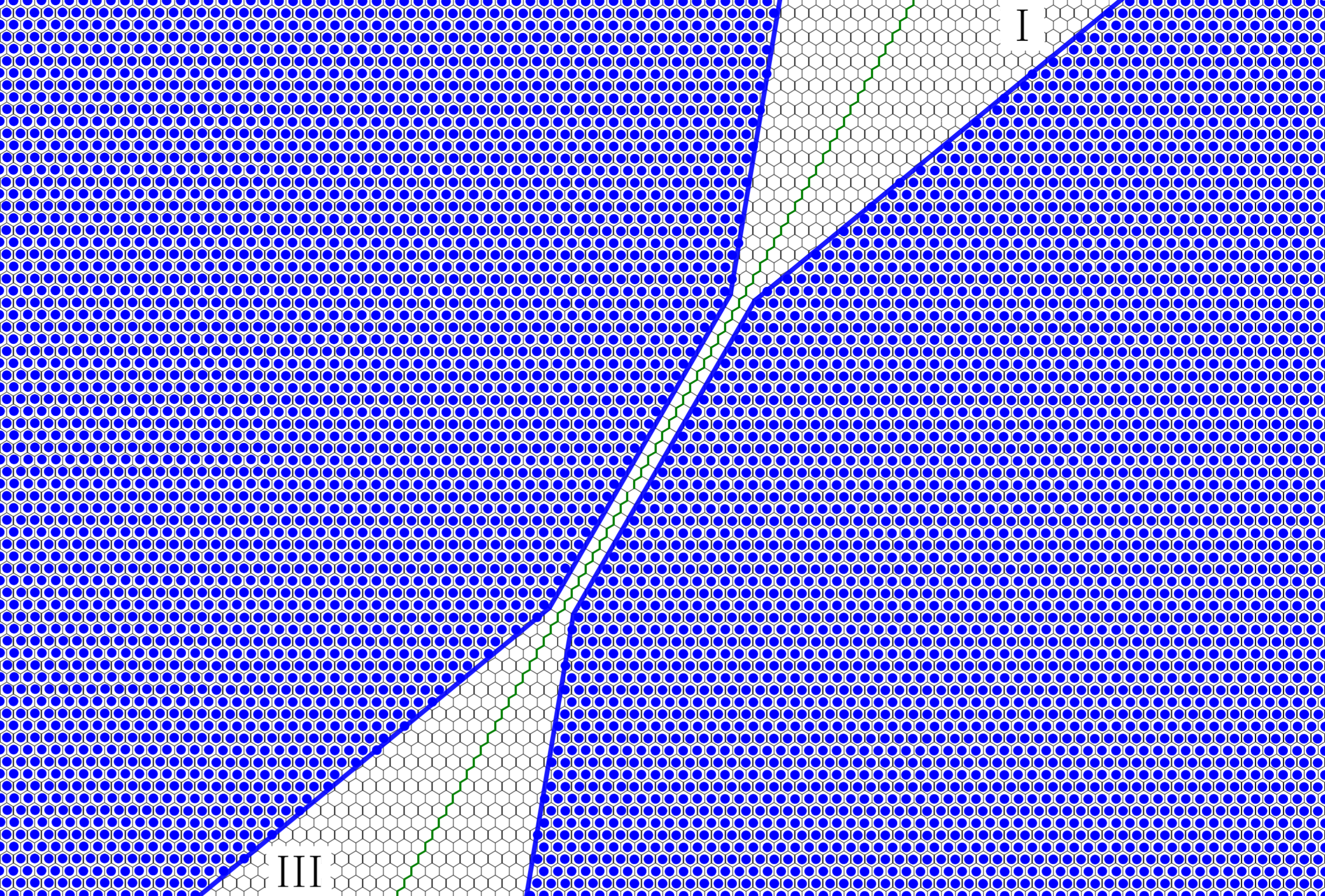}
\captionsetup{style=rightside,font=footnotesize}
\caption{Biased initial position}
\label{Figure.2-for-1-cone}
\end{wrapfigure}
In particular, we will describe a strategy that allows Red to construct a winning $\mathbb{Z}$-chain in exactly $\omega$ turns when starting from a board heavily biased in favor of Blue, as the one shown in Figure~\ref{Figure.2-for-1-cone}---it will follow that Red can win two-for-one infinite Hex starting from the empty board, even without playing first.

~

\begin{theorem}\label{Theorem.2-for-1-win}
In the infinite Hex variation in which Red places two
stones on each turn, while Blue places only one, Red has a winning strategy. Indeed, Red can win even when most of the board is given over to Blue as shown in Figure~\ref{Figure.2-for-1-cone}, where the narrow center channel can be as long as desired and the angle of the two cones on each end can be as tight as desired. Furthermore, Red can win on such boards with two-for-one play, but infinitely often electing to play only one-for-one.
\end{theorem}

\enlargethispage{20pt}
\medskip\noindent\emph{Proof.} 
Red's winning strategy in this two-for-one variation of the game will be a simple modification of the strategy used in the three-for-one case of Theorem~\ref{Theorem.3-for-1}. Namely, as before Red will pick an initial goal diagonal as indicated in green in Figure~\ref{Figure.2-for-1-cone}, aiming to construct a winning path along it, with the difference here being that in order to get away with only two moves each round instead of three, Red will be willing at times to allow small deviations in the goal path. These 
deviations, which will occur only infrequently and be spaced far apart, will save a move when they occur and thereby enable Red to make progress on building the winning path.\goodbreak

\begin{wrapfigure}{r}{.45\textwidth}
\hfill
\includegraphics[width=.42\textwidth]{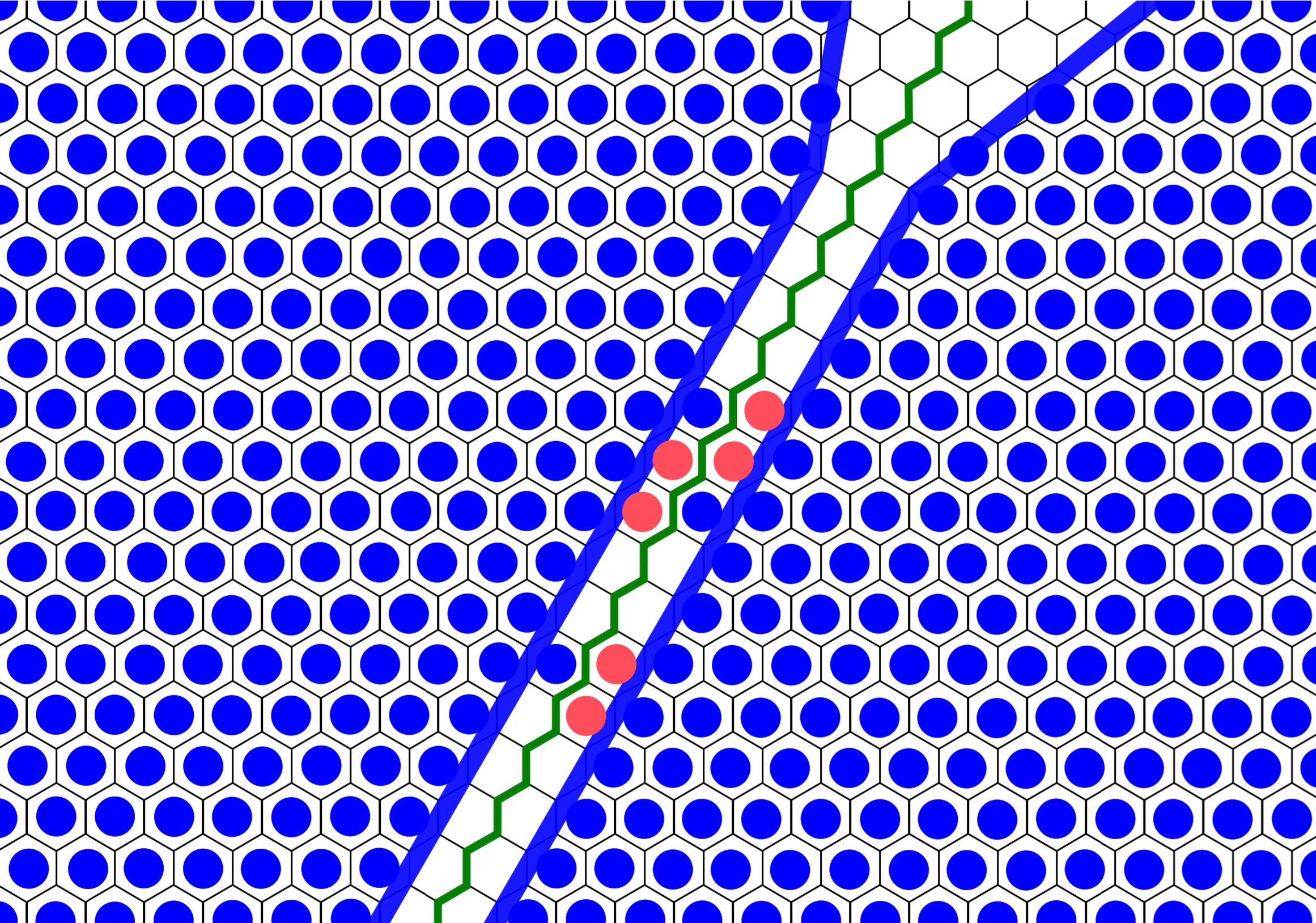}
\captionsetup{style=rightside,font=footnotesize}
\caption{Play in the center channel}
\label{Figure.2-for-1-cone-play-in-channel}
\end{wrapfigure}
Red aims of course to build a connected $\mathbb{Z}$-chain of Red stones following the green diagonal path, thereby constructing a winning play. The main part of the strategy calls for Red as before to respond to attacks on the target diagonal path with two-for-one adjacent Red stones directly across the path. For example, whenever Blue plays in the center channel region, Red will respond with adjacent two-for-one play as shown in Figure~\ref{Figure.2-for-1-cone-play-in-channel}. This manner of play in the channel will prevent Blue from ever bridging that narrow divide. 

\noindent Similarly, during normal play Red will also respond to attacks on the target diagonal out in the open conical region with adjacent two-for-one play, as shown in Figure~\ref{Figure.2-for-1-cone-reply-with-2}. This manner of play will prevent Blue from ever cutting across the target diagonal. 

Furthermore, if Blue should ever happen to place a stone not directly adjacent to the current target diagonal, then Red will breathe a sigh of relief, taking the opportunity to ignore the move and instead place Red stones so as to make progress on the desired winning diagonal path, completing it outward from the center. Such moves by Blue give Red a free opportunity for progress.

\begin{wrapfigure}{l}{.45\textwidth}
\includegraphics[width=.42\textwidth]{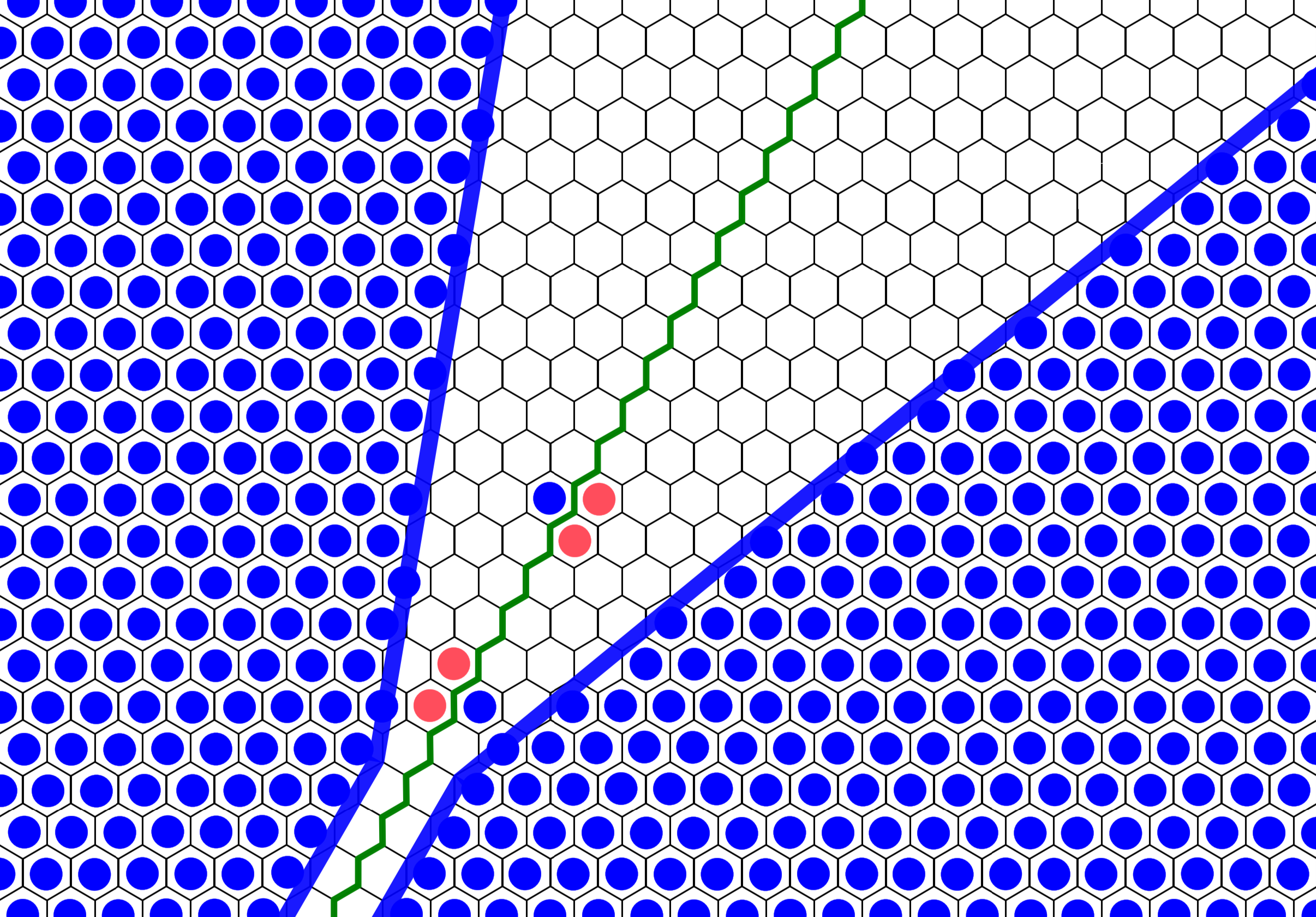}
\hfill
\captionsetup{style=leftside,font=footnotesize}
\caption{Red replies with two stones}
\label{Figure.2-for-1-cone-reply-with-2}
\end{wrapfigure}
But of course, Red cannot count on such Blue moves occurring, and yet still Red wants to ensure that the winning path will be completed by time $\omega$. It is for this reason that Red adopts the path-bending revision strategy we shall now describe. Namely, whenever Blue attacks the current target path at an extreme point far out enough in the open conical region, meaning that the blue stone is placed adjacent to the target path out in the open conical region, further out along it than any stone previously placed and far enough that the revised path will remain within the conical region, then Red will opt to respond with only one stone immediately adjacent to it, while also bending the target path around that stone as indicated in Figure~\ref{Figure.2-for-1-cone-reply-with-twist}. With the other stone, meanwhile, Red will make progress on the winning path by completing it outward from the center---this is exactly the advantage that Red needs to win by time $\omega$. Notice that the center part of the target path is progressively stabilized with each subsequent bending, since the bends occur only increasingly far out on the path.

\begin{wrapfigure}{r}{.45\textwidth}
\vskip-2ex
\hfill
\includegraphics[width=.42\textwidth]{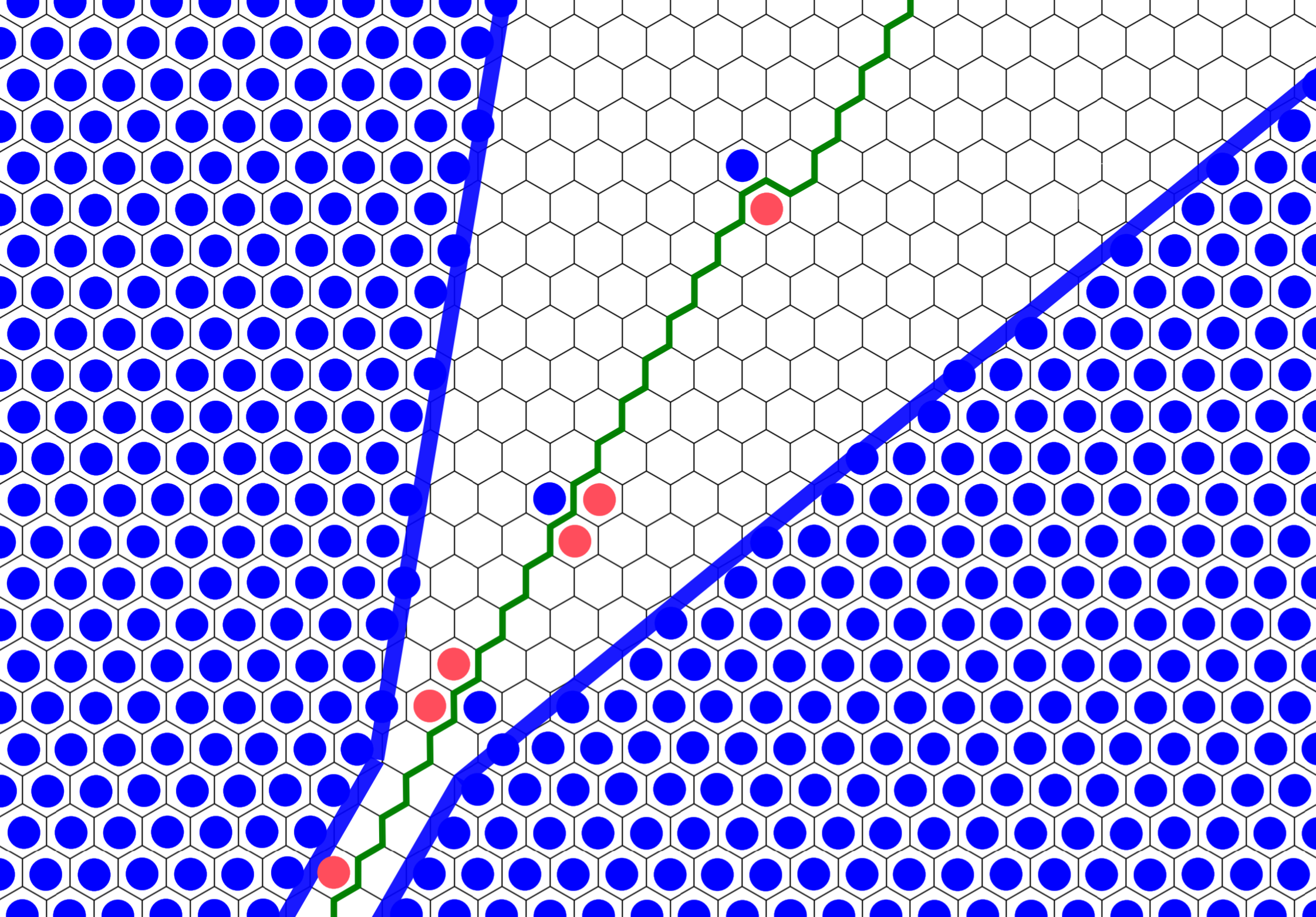}
\captionsetup{style=leftside,font=footnotesize}
\caption{Path-bending response, saving a Red stone}
\label{Figure.2-for-1-cone-reply-with-twist}
\end{wrapfigure}
Let us argue now that this is a winning strategy for Red, who will complete the winning $\mathbb{Z}$-chain by time $\omega$. We have already noted that Blue will be unable ever to bridge the narrow channel gap or to cut across the target diagonal path. What remains is for us to show that in any infinite play, Red will have infinitely many opportunities to make progress on constructing the winning path from the center outward. The reason for this is that after any finite play, there are only finitely many Blue moves available that will require immediate two-for-one response at that location. These are the blue moves attacking the current target path, but not at an extreme point that would invoke the path-bending response. Once all those locations are filled, any Blue move will either not attack the target path at all or will attack it at an extreme location that allows for the path-bending response. In both of these situations, Red is able to make free progress on the winning path, and so there will be infinitely many opportunities to do so. We conclude that Red will eventually complete a winning $\mathbb{Z}$-chain in a standard play of order type $\omega$.

Notice that the argument would work equally well if Red opted to play only one-for-one with every other path-bending move  and similarly opt for a one-move response whenever Blue makes a nonattacking off-path  move. Since such moves must occur infinitely often in any infinite play, this shows that infinitely often Red can respond with only one stone, proving the final claim of the theorem. 
\hfill\fbox{}\medskip\goodbreak 


Thus, the answer to Question~\ref{Question.Multiple-stones} is negative. The argument does not seem to succeed in the variation of infinite Hex in which Red plays two stones only on infinitely many turns chosen by Blue, or even on turns predetermined by adding some rule, starting from an empty board. The strategy described above requires Red to place two stones when Blue makes certain critical moves, such as those in the central channel or close enough to the center on the target green path.

In general, we can consider $m$-for-$n$ variations of infinite Hex in which Red and Blue place, respectively, $m$ and $n$ stones at each turn, starting from an empty board. The surrounding strategy allows Red to force a win in the $(6n$$+$$1)$-for-$n$ case---this can be refined into an half-surrounding strategy to achieve a win in the $(4n$$+$$1)$-for-$n$ variation. Notice that the dynamics of the game changes significantly when both players place more than one stone per turn, so that the mirroring strategy of Theorem \ref{Theorem.main} cannot be used to show whether the two-for-two variation of infinite Hex is still a draw. However, we conjecture that $n$-for-$n$ infinite Hex is a draw also for $n>1$.

\section{Open questions}

We conclude by briefly collecting together here several questions we have about finite and infinite Hex that remain open.

\renewcommand\labelenumi{\theenumi.}
\begin{enumerate}

\item Is infinite Hex a draw when starting from a board with finitely many stones already placed? (Conjecture~\ref{Conjecture.Finite-advantage})

\item Can Blue win Hex played on arbitrarily large trapezoids? (see Theorem~\ref{Theorem.If-trapezoids-then-draw})

\item Equivalently, can Blue win infinite Hex played on all truncated quadrants? (Question~\ref{Question.Truncated-quadrants})

\item Can Red win infinite Hex when starting with the advantage of an infinite horizontal red line? And with an infinite red line with negative slope? And with a horizontal and `vertical' line together? (Question~\ref{Question.Infinite-advantage})

\item Can Red win infinite Hex when starting with the advantage of two red quadrants? (Question~\ref{Question.Quadrant-advantage})

\item Can Red win infinite Hex in standard play time $\omega$ if infinitely often, on turns chosen by Blue (or determined by some rule), Red is allowed to place two stones? (see Theorem~\ref{Theorem.2-for-1-win})

\item For what values of $m$ is the $m$-for-$n$ variant of infinite Hex a draw? In particular, is the $n$-for-$n$ variant a draw? (see Theorem~\ref{Theorem.2-for-1-win})

\item If Red wins a play of infinite Hex, must there be a winning red $\mathbb{Z}$-chain that is geodesic, that is, a winning chain that also constitutes a minimal-length red connection between any two hex tiles appearing in it? 

\item What is the game value of finite Hex played on a $n\times n$ rhombus? What bounds can be established? In light of the fact that strong players often nearly fill the finite Hex board, is the game value asymptotically $n^2/2$?

\end{enumerate}

\printbibliography

\end{document}